\documentclass[11pt]{article}

\usepackage{amsfonts,amssymb,amsmath,amsthm,epsfig,euscript,epstopdf}
\usepackage[noadjust]{cite}
\usepackage{filecontents}
\newtheorem{theorem}{Theorem}

\long\def\symbolfootnote[#1]#2{\begingroup
\def\thefootnote{\fnsymbol{footnote}}\footnote[#1]{#2}\endgroup}

\newcommand{\qbinom}[2]{\genfrac{[}{]}{0pt}{}{#1}{#2}_{q}}

\newcommand{\cref}[1]{Corollary \ref{corollary:#1}}

\newcommand{\fig}[2]{\begin{figure}[ht]
\centerline{\scalebox{.66}{\epsfig{file=#1.eps}}}
\caption{#2}
\label{fig:#1}
\end{figure}}

\title{A Fibonacci analogue of Stirling numbers}

\author{
Quang T. Bach \\
\small Department of Mathematics\\[-0.8ex]
\small University of California, San Diego\\[-0.8ex]
\small La Jolla, CA 92093-0112. USA\\[-0.8ex]
\small \texttt{qtbach@ucsd.edu}
\and
Roshil Paudyal\\
\small{Department of Mathematics} \\
\small{Howard University}\\
\small \texttt{roshil.paudyal@bison.howard.edu}
\and
Jeffrey B. Remmel \\
\small Department of Mathematics\\[-0.8ex]
\small University of California, San Diego\\[-0.8ex]
\small La Jolla, CA 92093-0112. USA\\[-0.8ex]
\small \texttt{remmel@math.ucsd.edu}
\and
}

\date{\small Submitted: Date 1;  Accepted: Date 2;
 Published: Date 3.\\
\small MR Subject Classifications: 05A15, 05E05 \\
keywords: Fibonacci numbers, Stirling numbers, Lah numbers}

\begin{document}

\maketitle

\begin{abstract}
In recent years, there has been renewed interest in 
the Fibonomials $\binom{n}{k}_F$.  That is, we define the Fibonacci numbers 
by setting $F_1=1=F_2$ and $F_n =F_{n-1}+F_{n-2}$ for $n \geq 3$. 
We let $n_F! = F_1 \cdots F_n$ and 
$\binom{n}{k}_F = \frac{n_F!}{k_F!(n-k)_F!}$.  One can 
easily prove that $\binom{n}{k}_F$ is an integer and 
a combinatorial interpretation of $\binom{n}{k}_F$ was given in \cite{SS}. 

The goal of this paper is to find similar analogues for 
the Stirling numbers of the first and second kind and the Lah numbers. 
That is, we let $(x)_{\downarrow_0} = (x)_{\uparrow_0} = 1$ and for 
$k \geq 1$, $(x)_{\downarrow_k} = x(x-1) \cdots (x-k+1)$ and 
$(x)_{\uparrow_k} = x(x+1) \cdots (x+k-1)$. Then the Stirling numbers of 
the first and second kind are the connections coefficients 
between the usual power basis $\{x^n:n \geq 0\}$ and 
the falling factorial basis $\{(x)_{\downarrow_n}:n \geq 0\}$ in 
the polynomial ring $\mathbb{Q}[x]$ and the Lah numbers  
are the connections coefficients 
between the rising factorial  basis $\{(x)_{\uparrow_n}:n \geq 0\}$ and 
the falling factorial basis $\{(x)_{\downarrow_n}:n \geq 0\}$ in 
the polynomial ring $\mathbb{Q}[x]$.  In 
this paper, we will focus the Fibonacci analogues of the Stirling 
numbers.  Our idea is to 
replace the falling factorial basis and the rising factorial 
basis by the Fibo-falling factorial basis 
$\{(x)_{\downarrow_{F,n}}:n \geq 0\}$ and the Fibo-rising factorial 
basis $\{(x)_{\uparrow_{F,n}}:n \geq 0\}$ where 
$(x)_{\downarrow_{F,0}} = (x)_{\uparrow_{F,0}} = 1$ and for 
$k \geq 1$, $(x)_{\downarrow_{F,k}} = x(x-F_1) \cdots (x-F_{k-1})$ and 
$(x)_{\uparrow_{F,k}} = x(x+F_1) \cdots (x+F_{k-1})$. Then we study the 
combinatorics of the connection coefficients between
the usual power basis, the Fibo-falling factorial basis, and 
the Fibo-rising factorial basis. In each case, we can give 
a rook theory model for the connections coefficients 
and show how this rook theory model can give combinatorial 
explanations for many of the properties of these coefficients. 

\end{abstract}

\section{Introduction}

In 1915, Fonten\'e in \cite{Fon} suggested a generalization 
$n_A!$ and $\binom{n}{k}_A$ of $n!$ and the binomial coefficient  
$\binom{n}{k}$  depending on 
any sequence $A = \{A_n:n \geq 0\}$ of real or complex numbers such 
that $A_0 =0$ and $A_n \neq 0$ for all $n \geq 1$ by defining 
$0_A ! = 1$, $n_A! =  A_1 A_2 \cdots A_n$ for $n \geq 1$, and 
$\binom{n}{k}_A = \frac{n_A!}{k_A! (n-k)_A!}$ for $0 \leq k \leq n$. 
If we let $F = \{F_n: n\geq 0\}$ be the sequence of Fibonacci numbers 
defined by $F_0 =0$, $F_1 =1$, and $F_n = F_{n-1}+F_{n-2}$ 
for $n \geq 2$, then the $\binom{n}{k}_F$'s are known as the Fibonomials. 
The Fibonomial $\binom{n}{k}_F$ are integers since Gould 
\cite{Gould} proved that 
$$\binom{n}{k}_F = F_{k+1} \binom{n-1}{k}_F +F_{n-k+1} 
\binom{n-1}{k-1}_F.$$
There have been a series of papers looking at the 
properties of the Fibonomials \cite{ACMS,CS,Gould,GS,H,Torreto,Troj,SS}.
There is a nice combinatorial model for 
the Fibonomial coefficients given in \cite{SS} and 
the papers \cite{ACMS,CS} have developed further properties and 
generalizations of the Fibonomials using this model.

Our goal in this paper is to define Fibinomials type analogues 
for the Stirling numbers of the first and second kind and 
to study their properties. 
Let $\mathbb{Q}$ denote the rational numbers and 
$\mathbb{Q}[x]$ denote the ring of polynomials over $\mathbb{Q}$.
There are three very natural bases for $\mathbb{Q}[x]$. 
The usual power basis $\{x^n: n\geq 0\}$, 
the falling factorial basis  $\{(x)_{\downarrow_n}: n\geq 0\}$, 
and the rising factorial basis  $\{(x)_{\uparrow_n}: n\geq 0\}$. 
Here we let $(x)_{\downarrow_0} = (x)_{\uparrow_0} = 1$ and for 
$k \geq 1$, $(x)_{\downarrow_k} = x(x-1) \cdots (x-k+1)$ and 
$(x)_{\uparrow_k} = x(x+1) \cdots (x+k-1)$.
Then the Stirling numbers of the first kind $s_{n,k}$, 
the Stirling numbers of the second kind $S_{n,k}$ and 
the Lah numbers $L_{n,k}$ are defined by specifying 
that for all $n \geq 0$, 
\begin{equation*}
(x)_{\downarrow_n} = \sum_{k=1}^n s_{n,k} \ x^k, \ \ 
x^n = \sum_{k=1}^n S_{n,k} \ (x)_{\downarrow_k}, \ \mbox{and} \ \ 
(x)_{\uparrow_n} = \sum_{k=1}^n L_{n,k} \ (x)_{\downarrow_k}.
\end{equation*}
The signless Stirling numbers of the first kind 
are defined by setting $c_{n,k} = (-1)^{n-k} s_{n,k}$. 
Then it is well known that $c_{n,k}$, $S_{n,k}$, and $L_{n,k}$ 
can also be defined by the recursions that 
$c_{0,0} = S_{0,0} = L_{0,0} = 1$, 
$c_{n,k} = S_{n,k} = L_{n,k} = 0$ if either $n < k$ or 
$k < 0$, and 
\begin{equation*}
c_{n+1,k} = c_{n,k-1}+ n c_{n,k}, \ \ 
S_{n+1,k} =  S_{n,k-1}+kS_{n,k}, \ \mbox{and} \ \  
L_{n+1,k} = L_{n,k-1} +(n+k) L_{n,k}
\end{equation*}
for all $n,k \geq 0$. 
There are well known combinatorial interpretations of 
these connection coefficients. That is, 
$S_{n,k}$ is the number of set partitions of $[n] = \{1, \ldots, n\}$ 
into $k$ parts, $c_{n,k}$ is the number of permutations 
in the symmetric group $S_n$ with $k$ cycles, and 
$L_{n,k}$ is the number of ways to place $n$ labeled balls into $k$ unlabeled 
tubes with at least one ball in each tube.

We start with the tiling model of the $F_n$ of \cite{SS}. 
That is, let $\mathcal{FT}_n$ denote the set of tilings  
a column of height $n$ with tiles of height  
1 or 2 such that bottom most tile is of height 1. 
For example, possible tiling configurations 
for $\mathcal{FT}_i$ for $i \leq 4$ are shown in 

\fig{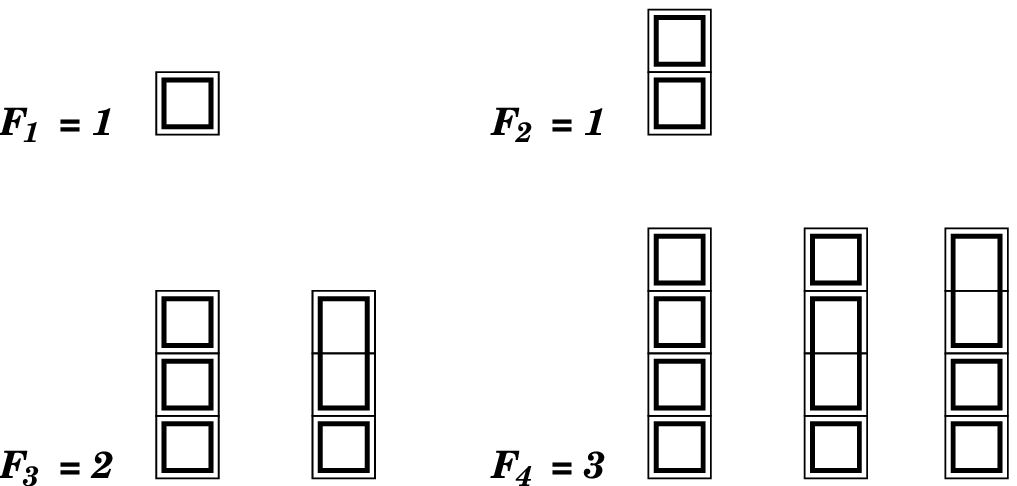}{The tilings counted by $F_i$ for $1 \leq i \leq 4$.}

For each tiling $T \in \mathcal{FT}_n$, we let 
$\mathrm{one}(T)$ is the number of tiles of 
height 1 in $T$ and $\mathrm{two}(T)$ is the number of tiles of 
height 2 in $T$ and define 
\begin{equation}
F_n(p,q) = \sum_{T \in \mathcal{FT}_n} q^{\mathrm{one}(T)}p^{\mathrm{two}(T)}.
\end{equation} 
It is easy to see that 
 $F_1(p,q) =q$, $F_2(p,q) =q^2$ and $F_n(p,q) =q F_{n-1}(p,q)+ 
pF_{n-2}(p,q)$ for $n \geq 2$ so that $F_n(1,1) =F_n$.
We then 
define the $p,q$-Fibo-falling 
factorial basis 
$\{(x)_{\downarrow_{F,p,q,n}}:n \geq 0\}$ and the $p,q$-Fibo-rising factorial 
basis $\{(x)_{\uparrow_{F,p,q,n}}:n \geq 0\}$ by setting 
$(x)_{\downarrow_{F,p,q,0}} = (x)_{\uparrow_{F,p,q,0}} = 1$ and setting
\begin{eqnarray}
 (x)_{\downarrow_{F,p,q,k}} &=& x(x-F_1(p,q)) \cdots (x-F_{k-1}(p,q)) \ \mbox{and} \\ 
(x)_{\uparrow_{F,p,q,k}} &=& x(x+F_1(p,q)) \cdots (x+F_{k-1}(p,q))
\end{eqnarray}
for $k \geq 1$. 

Our idea to define $p,q$-Fibonacci analogues of 
the Stirling numbers of the first kind, $\mathbf{sf}_{n,k}(p,q)$, 
the Stirling numbers of the second kind, $\mathbf{Sf}_{n,k}(p,q)$, 
and the Lah numbers, $\mathbf{Lf}_{n,k}(p,q)$, is to 
define them   
to be the connection coefficients between the usual power basis 
$\{x^n: n \geq 0\}$ and the $p,q$-Fibo-rising factorial and 
$p,q$-Fibo-falling factorial bases. 
That is, we define $\mathbf{sf}_{n,k}(p,q)$, $\mathbf{Sf}_{n,k}(p,q)$, 
and $\mathbf{Lf}_{n,k}(p,q)$ by the equations  
\begin{equation}\label{pqsfdef}
(x)_{\downarrow_{F,p,q,n}} = \sum_{k=1}^n \mathbf{sf}_{n,k}(p,q) \ x^k,
\end{equation}
\begin{equation}\label{pqSfdef}
x^n = \sum_{k=1}^n \mathbf{Sf}_{n,k}(p,q) \ (x)_{\downarrow_{F,p,q,k}},
\end{equation}
and 
\begin{equation}\label{pqLfdef}
(x)_{\uparrow_{F,p,q,n}} = \sum_{k=1}^n \mathbf{Lf}_{n,k}(p,q) \ (x)_{\downarrow_{F,p,q,k}}
\end{equation}
for all $n \geq 0$. 

It is easy to see that these equations imply simple recursions for 
the connection coefficients $\mathbf{sf}_{n,k}(p,q)$s, $\mathbf{Sf}_{n,k}(p,q)$s, and $\mathbf{Lf}_{n,k}(p,q)$s. 
That is, the $\mathbf{sf}_{n,k}(p,q)$s can be defined by the recursions  
\begin{equation}\label{pqsfrec}
\mathbf{sf}_{n+1,k}(p,q) = \mathbf{sf}_{n,k-1}(p,q)- F_n(p,q) \mathbf{sf}_{n,k}(p,q) 
\end{equation}  plus the boundary 
conditions $\mathbf{sf}_{0,0}(p,q)=1$ and $\mathbf{sf}_{n,k}(p,q) =0$ if $k > n$ or $k < 0$. 
The $\mathbf{Sf}_{n,k}(p,q)$s can be defined 
by the recursions   
\begin{equation}\label{pqSfrec}
\mathbf{Sf}_{n+1,k}(p,q) = \mathbf{Sf}_{n,k-1}(p,q)+ F_k(p,q) \mathbf{Sf}_{n,k}(p,q) 
\end{equation}
plus the boundary 
conditions $\mathbf{Sf}_{0,0}(p,q)=1$ and $\mathbf{Sf}_{n,k}(p,q) =0$ if $k > n$ or $k < 0$.
The $\mathbf{Lf}_{n,k}(p,q)$s can be defined 
by the recursions    
\begin{equation}\label{pqLfrec}
\mathbf{Lf}_{n+1,k}(p,q) = \mathbf{Lf}_{n,k-1}(p,q)+ (F_k(p,q) +F_n(p,q))\mathbf{Lf}_{n,k}(p,q) 
\end{equation}
plus the boundary 
conditions $\mathbf{Lf}_{0,0}(p,q)=1$ and $\mathbf{Lf}_{n,k}(p,q) =0$ if $k > n$ or $k < 0$.
If we define $\mathbf{cf}_{n,k}(p,q):= (-1)^{n-k}\mathbf{sf}_{n,k}(p,q)$, then 
$\mathbf{cf}_{n,k}(p,q)$s can be defined by the recursions   
\begin{equation}\label{pqcfrec}
\mathbf{cf}_{n+1,k}(p,q) = \mathbf{cf}_{n,k-1}(p,q)+F_n(p,q) \mathbf{cf}_{n,k}(p,q) 
\end{equation}  plus the boundary 
conditions $\mathbf{cf}_{0,0}(p,q)=1$ and $\mathbf{cf}_{n,k}(p,q) =0$ if $k > n$ or $k < 0$. It also follows that 
\begin{equation}\label{pqcfdef}
(x)_{\uparrow_{F,p,q,n}} = \sum_{k=1}^n \mathbf{cf}_{n,k}(p,q) \ x^k.
\end{equation}

The goal of this paper is to develop a combinatorial model 
for the Fibo-Stirling numbers $\mathbf{Sf}_{n,k}$ and $\mathbf{cf}_{n,k}$. 
Our combinatorial model is 
a modification of the rook theory model for $S_{n,k}$ 
and $c_{n,k}$ except that we replace rooks by Fibonacci 
tilings.  We will show that we can use 
this model to give combinatorial proofs 
of the recursions (\ref{pqcfrec}) and (\ref{pqSfrec}) 
and defining equations (\ref{pqcfdef}) and (\ref{pqSfdef}) 
as well as a combinatorial proof of the fact that 
the infinite matrices $||\mathbf{Sf}_{n,k}||_{n,k \geq 0}$ 
and $||\mathbf{sf}_{n,k}||_{n,k \geq 0}$ are inverses of each other. 
There is also a rook theory model for the 
$\mathbf{Lf}_{n,k}(p,q)$s, but it is significantly different 
from the rook theory model  for the  
$\mathbf{Sf}_{n,k}(p,q)$s and the $\mathbf{cf}_{n,k}(p,q)$s. Thus we 
will give such a model in a subsequent paper. We should also 
note that there is a more general rook theory model which can be used 
to give combinatorial interpretations 
for the coefficients $\mathbf{sf}_{n,k}(1,1)$,  $\mathbf{Sf}_{n,k}(1,1)$, 
$\mathbf{Lf}_{n,k}(1,1)$, 
and $\mathbf{cf}_{n,k}(1,1)$ due to Miceli and the third author \cite{MR}. 
However, that model does not easily adapt to give a rook 
theory model for  the coefficients 
$\mathbf{sf}_{n,k}(p,q)$,  $\mathbf{Sf}_{n,k}(p,q)$, $\mathbf{Lf}_{n,k}(p,q)$, 
and $\mathbf{cf}_{n,k}(p,q)$.

The outline of this paper is as follows. 
In Section 2, we shall give a general rook theory model 
of tiling placements on Ferrers boards $B$. In the special 
case where $B$ is the board whose column 
heights are $0,1, \ldots, n-1$, reading from left to 
right, our rook theory model will give us combinatorial interpretations 
for the $\mathbf{Sf}_{n,k}(p,q)$s and the $\mathbf{cf}_{n,k}(p,q)$s.  
  We shall develop general recursions 
for the analogue of file and rook numbers in this model which will 
specialize to give combinatorial proofs 
of the recursions (\ref{pqcfrec}) and (\ref{pqSfrec}). 
Similarly, we shall give combinatorial 
proofs of two general product formulas in
this model which will specialize to give combinatorial proofs of  
(\ref{pqcfdef}) and (\ref{pqSfdef}). In Section 
3, we shall give a combinatorial proof that 
the infinite matrices $||\mathbf{Sf}_{n,k}(p,q)||_{n,k \geq 0}$ 
and $||\mathbf{sf}_{n,k}(p,q)||_{n,k \geq 0}$ are inverses of each other. 
In Section 4, we shall give various generating functions 
and identities for the $\mathbf{Sf}_{n,k}(p,q)$s and the $\mathbf{cf}_{n,k}(p,q)$s.

\section{A rook theory model for the 
$\mathbf{Sf}_{n,k}(p,q)$s and the $\mathbf{cf}_{n,k}(p,q)$s.}

In this section, we shall develop a new type of rook theory 
model which is appropriate to interpret 
the $\mathbf{Sf}_{n,k}(p,q)$s and the $\mathbf{cf}_{n,k}(p,q)$s.  
A Ferrers board $B=F(b_1, \ldots, b_n)$ is 
a board whose column heights are $b_1, \ldots, b_n$, reading 
from left to right, such that $0\leq b_1 \leq b_2 \leq \cdots \leq b_n$. 
We shall let $B_n$ denote the Ferrers board $F(0,1, \ldots, n-1)$. 
For example, the Ferrers board $B = F(2,2,3,5)$ is 
pictured on the left of Figure \ref{fig:Ferrers} 
and the Ferrers board $B_4$ is pictured on the right of 
Figure \ref{fig:Ferrers}. 

\fig{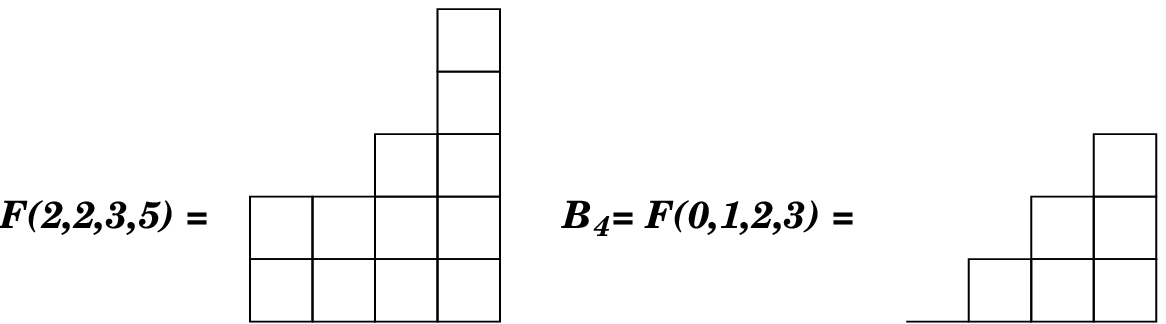}{Ferrers boards.}

Classically, there are two type of rook placements that we 
consider on a Ferrers board $B$.  First we let 
$\mathcal{N}_k(B)$ be the set of all placements of 
$k$ rooks in $B$ such that no two rooks lie in the same 
row or column.  We shall call an element  of 
$\mathcal{N}_k(B)$ a placement of $k$ non-attacking rooks 
in $B$ or just a rook placement for short.  We let 
$\mathcal{F}_k(B)$ be the set of all placements of 
$k$ rooks in $B$ such that no two rooks lie in the same 
column.  We shall call an element  of 
$\mathcal{F}_k(B)$ a file placement of $k$ rooks 
in $B$.   
Thus file placements differ from rook placements 
in that file placements allow 
two rooks to be in the same row.  For example, 
we exhibit a placement of 3 non-attacking rooks 
in $F(2,2,3,5)$ on the left in Figure 
\ref{fig:placements} and a file placement of 3 rooks on 
the right in  Figure \ref{fig:placements}.

\fig{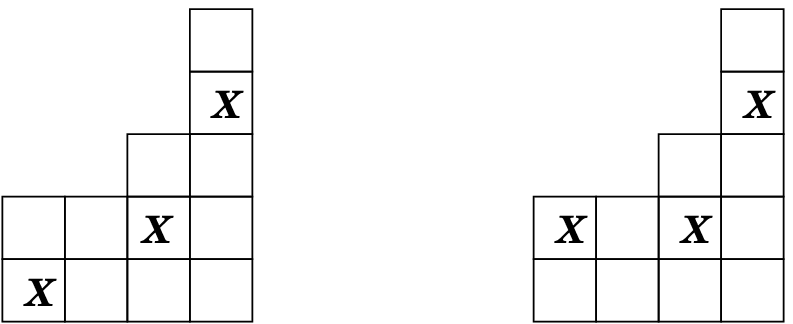}{Examples of rook and file placements.}

Given a Ferrers board $B = F(b_1, \ldots, b_n)$, we 
define the $k$-th rook number of $B$ to be 
$r_k(B) = |\mathcal{N}_k(B)|$ and the $k$-th file number 
of $B$ to be $f_k(B) = |\mathcal{F}_k(B)|$. Then the rook 
theory interpretation of the classical Stirling numbers is 
\begin{eqnarray*}
S_{n,k} &=& r_{n-k}(B_n) \ \mbox{for all}\ 1 \leq k \leq n \ \mbox{and} \\
c_{n,k} &=& f_{n-k}(B_n) \ \mbox{for all}\ 1 \leq k \leq n.
\end{eqnarray*}

Our idea is to modify the sets $\mathcal{N}_k(B)$ and 
$\mathcal{F}_k(B)$ to replace rooks with Fibonacci tilings. 
The analogue of file placements is very straightforward. 
That is, if $B=F(b_1, \ldots, b_n)$, then we let 
$\mathcal{FT}_k(B)$ denote the set of all configurations such that 
there are $k$ columns $(i_1, \ldots, i_k)$ of $B$ where 
$1 \leq i_1 < \cdots < i_k \leq n$ such that in each 
column $i_j$, we have placed one of the tilings $T_{i,j}$ for the Fibonacci 
number $F_{b_{i_j}}$.  We shall call such a configuration 
a Fibonacci file placement and denote it by 
$$P = ((c_{i_1},T_{i_1}),  \ldots, (c_{i_k},T_{i_k})).$$
 Let 
$\mathrm{one}(P)$ denote the number of tiles of height 1 that appear 
in $P$ and $\mathrm{two}(P)$ denote the number of tiles of height 2 that appear 
in $P$.  We then define the weight of $P$, $WF(P,p,q)$, to be 
$q^{\mathrm{one}(P)}p^{\mathrm{two}(P)}$.  For example, we have 
pictured an element $P$ of $\mathcal{FT}_3(F(2,3,4,4,5))$ in 
Figure \ref{fig:Fibfile} whose weight is $q^7 p^2$.

\fig{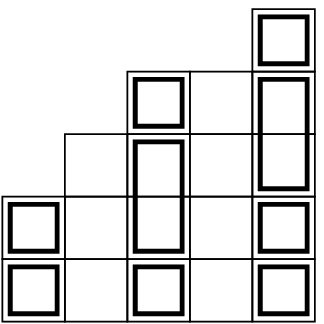}{A Fibonacci file placement.}

We define the $k$-th $p,q$-Fibonacci file polynomial of $B$, $\mathbf{fT}_k(B,p,q)$, 
by setting 
\begin{equation}
 \mathbf{fT}_k(B,p,q) = \sum_{P \in \mathcal{FT}_k(B)} WF(P,p,q).
\end{equation}
If $k =0$, then the only element of $\mathcal{FT}_k(B)$ is the empty placement 
whose weight by definition is 1.

Then we have the following two theorems concerning 
Fibonacci file placements in Ferrers boards. 

\begin{theorem}\label{thm:Ferrersfilerec} 
Let $B =F(b_1, \ldots, b_n)$ be a Ferrers 
board where $0 \leq b_1 \leq \cdots \leq b_n$ and $b_n > 0$. 
Let $B^- = F(b_1, \ldots, b_{n-1})$. Then for all 
$1 \leq k \leq n$, 
\begin{equation}\label{Fibtilerec}
\mathbf{fT}_k(B,p,q) = \mathbf{fT}_{k}(B^-,p,q)+ F_{b_n}(p,q) \mathbf{fT}_{k-1}(B^-,p,q).
\end{equation}
\end{theorem}
\begin{proof}
It is easy to see that the right-hand side (\ref{Fibtilerec}) 
is just the result of classifying the Fibonacci 
file placements $P$ in  $\mathcal{FT}_k(B)$ by whether there is a tiling  
in the last column. If there is no tiling in the last column of $P$, 
then removing the last column of  $P$ produces 
 an element of $\mathcal{FT}_k(B^-)$. Thus such placements 
contribute $\mathbf{fT}_{k}(B^-,p,q)$ to $\mathbf{fT}_{k}(B,p,q)$. 
If there is a tiling in 
the last column, then the Fibonacci file placement  
that results by removing the last column is an 
element of $\mathcal{FT}_{k-1}(B^-)$ and the sum of 
the weights of the possible 
Fibonacci tilings of height $b_n$ for the last column 
is $F_{b_n}(p,q)$. Hence such placements 
contribute $F_{b_n}(p,q) \mathbf{fT}_{k-1}(B^-,p,q)$ to 
$\mathbf{fT}_{k}(B,p,q)$. 
\end{proof}

If $B=F(b_1, \ldots, b_n)$ is a Ferrers board, 
then we let $B_x$ denote the board that results by 
adding $x$ rows of length $n$ below $B$.  We label 
these rows from top to bottom with the numbers 
$1,2, \ldots, x$. We shall call  
the line that separates $B$ from these $x$ rows the {\em bar}. 
A mixed file placement $P$ on the board $B_x$ consists 
of picking for each column $b_i$ either (i) a Fibonacci tiling 
$T_i$ of height $b_i$ above the bar or (ii) picking 
a row $j$ below the bar to place a rook in the cell in row $j$ 
and column $i$.  Let $\mathcal{M}_n(B_x)$ denote set of all 
mixed rook placements on $B$. For any $P \in \mathcal{M}_n(B_x)$, 
we let $\mathrm{one}(P)$ denote the number of tiles of height 1 that appear 
in $P$ and $\mathrm{two}(P)$ denote the set tiles of height 2 that appear 
in $P$.  We then define the weight of $P$, $WF(P,p,q)$, to be 
$q^{\mathrm{one}(P)}p^{\mathrm{two}(P)}$.
For example, 
Figure \ref{fig:mixed} pictures a mixed placement $P$ in 
$B_x$ where $B = F(2,3,4,4,5,5)$ and $x$ is 9 such that 
$WF(P,p,q) =  q^7p^2$. 

\fig{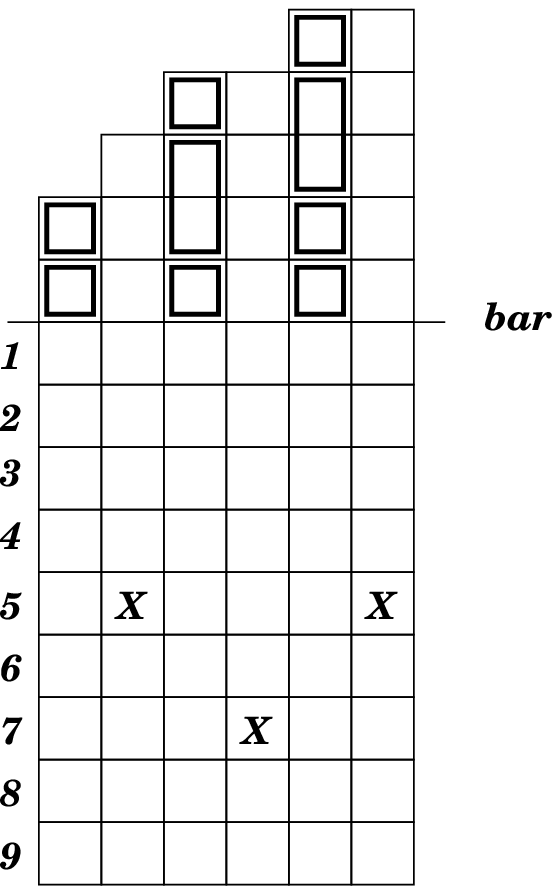}{A mixed file placement.}

Our next theorem results from counting 
$\sum_{P \in \mathcal{M}_n(B_x)} WF(P,p,q)$ in two different ways. 

\begin{theorem}\label{thm:Ferrersfileprod}
Let $B =F(b_1, \ldots, b_n)$ be a Ferrers 
board where $0 \leq b_1 \leq \cdots \leq b_n$. 
\begin{equation}\label{Fibfileprod}
(x+F_{b_1}(p,q))(x+F_{b_2}(p,q)) \cdots (x+F_{b_n}(p,q)) =
\sum_{k=0}^n \mathbf{fT}_k(B,p,q) x^{n-k}.
\end{equation}
\end{theorem}
\begin{proof}
Since both sides of (\ref{Fibfileprod}) are polynomials of 
degree $n$, it is enough to show that there are 
$n+1$ different values of $x$ for which the two sides are equal. 
In fact, we will show that the two sides are equal for 
any positive integer $x$. 

Thus fix $x$ to be a positive integer and consider 
the sum $S=\sum_{P \in \mathcal{M}_n(B_x)} WF(P,p,q)$. 
It is clear that each column of $b_i$ of $B$ contributes 
a factor of $x+ F_{b_i}(p,q)$ to $S$ so that 
$$S = \prod_{i=1}^n (x +  F_{b_i}(p,q)).$$

On the other 
hand, suppose that we fix a Fibonacci file placement  
$P \in \mathcal{FT}_k(B)$. 
Then we want to compute $S_P = \sum_{Q \in \mathcal{M}_n(B), 
Q \cap B = P} WF(Q,p,q)$ which is the sum of $WF(Q,p,q)$ over all 
mixed placements $Q$ such that $Q$ intersect $B$ equals $P$. 
It it easy to see that such a $Q$ arises by choosing 
a rook to be placed below the bar for each column 
that does not contain a tiling.  Since there are $x^{n-k}$ ways to 
do this, it follows that 
$\displaystyle S_P =WF(P,p,q) x^{n-k}$. 
Hence it follows that 
\begin{eqnarray*}
S &=&  \sum_{k=0}^n \sum_{P \in \mathcal{FT}_k(B)} S_P \\
&=& \sum_{k=0}^n x^{n-k} \sum_{P \in \mathcal{FT}_k(B)} WF(P,p,q) \\
&=& \sum_{k=0}^n \mathbf{fT}_k(B,p,q) \ x^{n-k}.
\end{eqnarray*}
\end{proof}

We should note that neither the proof of Theorem 
\ref{thm:Ferrersfilerec} nor \ref{thm:Ferrersfileprod} depended 
on the fact that $b_1 \leq b_2 \leq \ldots \leq b_n$. Thus 
they hold for arbitrary sequences of non-negative integers 
$(b_1, \ldots, b_n)$.

Now consider the special case of the previous two theorems 
when $B_n = F(0,1,2, \ldots, n-1)$. Then (\ref{Fibtilerec}) implies 
that 
$$\mathbf{fT}_{n+1-k}(B_{n+1},p,q) = \mathbf{fT}_{n+1-k}(B_n,p,q) + F_n(p,q) 
\mathbf{fT}_{n-k}(B_n,p,q).$$
It then easily follows that for all $0 \leq k \leq n$, 
\begin{equation}\label{cnkpqinterp}
 \mathbf{cf}_{n,k}(p,q) = \mathbf{fT}_{n-k}(B_n,p,q).
\end{equation}
Note that $\mathbf{cf}_{n,0}(p,q) = 0$ for all $n \geq 1$ since 
there are no Fibonacci file placements in 
$\mathcal{FT}_n(B_n)$ since there are only $n-1$ non-zero columns. 
Moreover such a situation, we see that (\ref{Fibtilerec}) 
implies that 
\begin{equation*}
x(x+F_1(p,q))(x+F_2(p,q))\cdots (x+F_{n-1}(p,q)) = 
\sum_{k=1}^n \mathbf{cf}_{n,k}(p,q) \ x^k.
\end{equation*}
Thus we have given a combinatorial proof of 
(\ref{pqcfdef}).

Our Fibonacci analogue of rook  placements is a slight 
variation of Fibonacci file placements. The main 
difference is that each 
tiling will cancel some of the top most cells in each column 
to its right that has not been canceled by a tiling 
which is further to the left.  Our goal is to ensure 
that if we start with a Ferrers board $B =F(b_1, \ldots, b_n)$, 
our cancellation will ensure that the number of 
uncanceled cells in the empty columns are $b_1, \ldots, b_{n-k}$, 
reading from left to right. 
That is, if $B=F(b_1, \ldots, b_n)$, then we let 
$\mathcal{NT}_k(B)$ denote the set of all configurations such that 
that there are $k$ columns $(i_1, \ldots, i_k)$ of $B$ where 
$1 \leq i_1 < \cdots < i_k \leq n$ such that the 
following conditions hold. \begin{itemize}
\item[1.] In column $c_{i_1}$, we place a Fibonacci tiling 
$T_{i,1}$ of height $b_{i_1}$ and for each $j > i_1$, 
this tiling cancels the top $b_j-b_{j-1}$ cells at the top 
of column $j$. This cancellation has the effect of 
ensuring that the number of uncanceled cells in the columns 
without tilings at this point is 
$b_1, \ldots, b_{n-1}$, reading from left to right. 
\item[2.] In column $c_{i_2}$, our cancellation due to the tiling 
in column $i_1$ ensures that there are $b_{i_2-1}$ uncanceled 
cells in column $i_2$. Then we place a Fibonacci tiling 
$T_{i,2}$ of height $b_{i_2-1}$ and for each $j > i_2$, 
we cancel the top $b_{j-1}-b_{j-2}$ cells in column $j$ 
that has not been canceled by the tiling in column $i_1$. 
This cancellation has the effect of 
ensuring that the number of uncanceled cells in the columns 
without tilings at this point 
is $b_1, \ldots, b_{n-2}$, reading from left to right.
\item[3.] In general, when we reach column $i_s$, we assume 
that the cancellation due to the tilings in columns  
$i_1, \ldots, i_{j-1}$ ensure that the number of uncanceled 
cells in the columns without tilings is $b_1, \ldots, b_{n-(s-1)}$, 
reading from left to right. Thus there will be 
$b_{i_s -(s-1)}$ uncanceled cells in column $i_s$ at this point. 
Then we place a Fibonacci tiling 
$T_{i,s}$ of height $b_{i_s-(s-1)}$ and for each $j > i_s$, 
this tiling will 
cancel the top $b_{j-(s-1)}-b_{j-s}$ cells in column 
$j$ that has not been canceled by the tilings in 
columns $i_1, \ldots, i_{s-1}$. 
This cancellation has the effect of 
ensuring that the number of uncanceled cells in columns 
without tilings at this point 
is $b_1, \ldots, b_{n-s}$, reading from left to right.
\end{itemize}

We shall call such a configuration 
a Fibonacci rook placement and denote it by 
$$P = ((c_{i_1},T_{i_1}),  \ldots, (c_{i_k},T_{i_k})).$$ 
Let $\mathrm{one}(P)$ denote the number of tiles of height 1 that appear 
in $P$ and $\mathrm{two}(P)$ denote the number of tiles of height 2 that appear 
in $P$.  We then define the weight of $P$, $WF(P,p,q)$, to be 
$q^{\mathrm{one}(P)}p^{\mathrm{two}(P)}$.  For example, on the left in 
Figure \ref{fig:Fibrook}, we have 
pictured an element $P$ of $\mathcal{NT}_3(F(2,3,4,4,6,6))$ 
whose weight is $q^4 p^2$. In 
Figure \ref{fig:Fibrook}, we have 
indicated the cells canceled by the tiling in column 
$i$ by placing an $i$ in the cell.  
We note in the special case where $B = F(0,k,2k, \ldots, (n-1)k)$, then 
our cancellation scheme is quite simple. That is, each tiling 
just cancels the top $k$ cells in each column to its right which 
has not been canceled by tilings to its left. 
For example, on the right in 
Figure \ref{fig:Fibrook}, we have 
pictured an element $P$ of $\mathcal{NT}_3(F(0,1,2,3,4,5))$ 
whose weight is $q^6 p$. Again, we have 
indicated the canceled cells by the tiling in column 
$i$ by placing an $i$ in the cell.

\fig{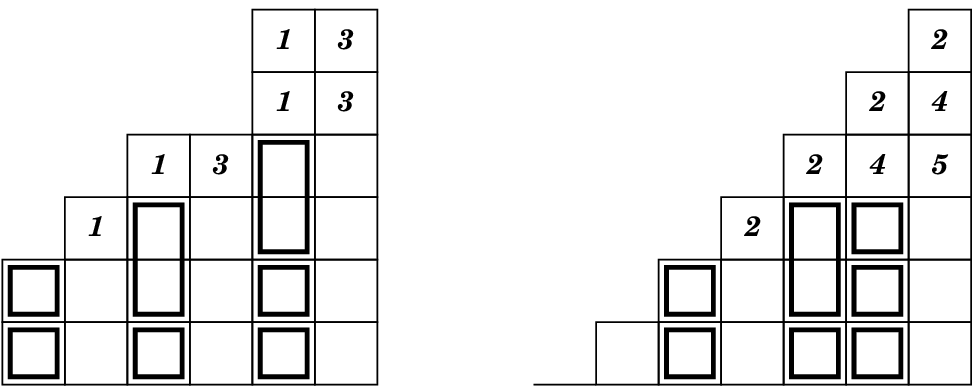}{A Fibonacci rook placement.}

We define the $k$-th $p,q$-Fibonacci rook polynomial of $B$, $\mathbf{rT}_k(B,p,q)$, 
by setting 
\begin{equation}
 \mathbf{rT}_k(B,p,q) = \sum_{P \in \mathcal{NT}_k(B)} WF(P,p,q).
\end{equation}
If $k =0$, then the only element of $\mathcal{FT}_k(B)$ is the empty placement 
whose weight by definition is 1.

Then we have the following two theorems concerning 
Fibonacci rook placements in Ferrers boards. 

\begin{theorem}\label{thm:Ferrersrookrec} 
Let $B =F(b_1, \ldots, b_n)$ be a Ferrers 
board where $0 \leq b_1 \leq \cdots \leq b_n$ and $b_n > 0$. 
Let $B^- = F(b_1, \ldots, b_{n-1})$. Then for all 
$1 \leq k \leq n$, 
\begin{equation}\label{Fibrookrec}
\mathbf{rT}_k(B,p,q) = \mathbf{rT}_{k}(B^-,p,q)+ F_{b_{n-(k-1)}}(p,q) \mathbf{rT}_{k-1}(B^-,p,q).
\end{equation}
\end{theorem}
\begin{proof}
It is easy to see that the right-hand side (\ref{Fibrookrec}) 
is just the result of classifying the Fibonacci 
rook placements $P$ in  $\mathcal{NT}_k(B)$ by whether there is a tiling  
in the last column. If there is no tiling in the last column of $P$, 
then removing the last column of $P$ gives 
 an element of $\mathcal{NT}_k(B^-)$. Thus such placements 
contribute $\mathbf{rT}_{k}(B^-,p,q)$ to $\mathbf{rT}_{k}(B,p,q)$. 
If there is a tiling in 
the last column, then the Fibonacci rook placement  
that results by removing the last column is an 
element of $\mathcal{NT}_{k-1}(B^-)$ and these tilings  
cancel the top $b_n-b_{n-(k-1)}$ cells of the last column. Then 
the weights of the possible 
Fibonacci tilings of height $b_{n-(k-1)}$ for the last column 
is $F_{b_{n-(k-1)}}(p,q)$. Hence such placements 
contribute $F_{b_{n-(k-1)}}(p,q) \mathbf{rT}_{k-1}(B^-,p,q)$ to $\mathbf{rT}_{k}(B,p,q)$. 
\end{proof}

We also have a product formula for Fibonacci rook 
placements in $B_n$. In this case, we have to use ideas from 
the proof of an even more general product formula 
due to Miceli and third author in \cite{MR}.

Let $B=F(b_1, \ldots, b_n)$ be  a Ferrers board and 
$x$ be a positive integer. 
Then we let $AugB_x$ denote the board where 
we start with $B_x$ and add the flip of the board $B$ about 
its baseline below the board. We shall call the 
the line that separates $B$ from these $x$ rows the {\em upper bar} 
and the line that separates the $x$ rows from the flip 
of $B$ added below the $x$ rows the {\em lower bar}. We shall 
call the flipped version of $B$ added below $B_x$ the board 
$\overline{B}$. For example, 
if $B=F(2,3,4,4,5,5)$, then the board $AugB_7$ is pictured 
in Figure \ref{fig:aug}.

\fig{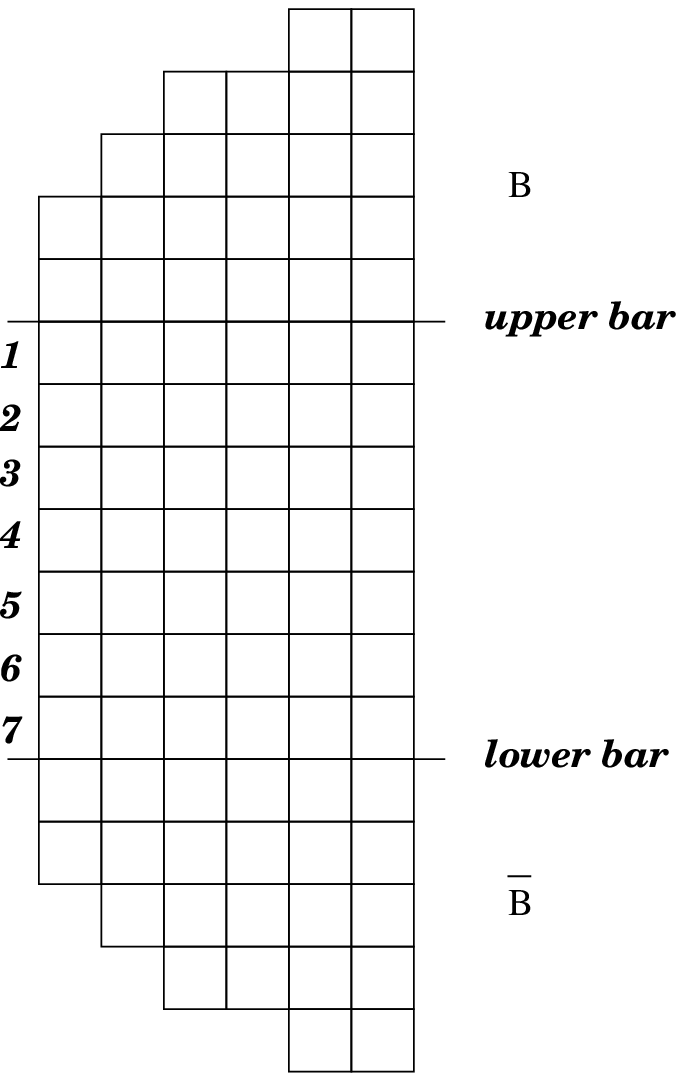}{An example of an augmented board $AugB_x$..}

The analogue of mixed placements in $AugB_x$ are 
more complex than the mixed placements for $B_x$. We process 
the columns from left to right. 
If we are in column 1, then we can do one of the following three things. \begin{itemize}
\item[i.] We can put a Fibonacci tiling in  
cells in column $b_1$ in $B$.  Then we must 
cancel the top-most cells in each of the columns in $B$ to its right 
so that the number of uncanceled cells in 
the columns to its right are $b_1,b_2, \ldots, b_{n-1}$, respectively, as 
we read from left to right. This means 
that we will cancel $b_i-b_{i-1}$ at the top of column $i$ in $B$ 
for $i=2, \ldots, n$. We also cancel the same number 
of cells at the bottom of the corresponding columns of $\overline{B}$.
\item[ii.] We can place a rook in any row of column $1$ that 
lies between the upper bar and lower bar.  This rook 
will not cancel anything. 
\item[iii.] We can put a flip of Fibonacci tiling 
in column $b_1$ of $\overline{B}$.  This tiling will not 
cancel anything. \end{itemize}

Next assume that when we get to column $j$, the 
number of uncanceled cells in the columns that have 
no tilings in $B$ and $\overline{B}$ are $b_1, \ldots, b_k$  for some $k$ as 
we read from left to right. Suppose there are 
$b_i$ uncanceled cells in $B$ in column $j$.  
Then we can do one of three things. \begin{itemize}
\item[i.] We can put a Fibonacci tiling of height $b_i$  in the uncanceled cells in column $j$ in $B$.  Then we must 
cancel top-most cells of the columns in $B$ to its right 
so that the number of uncanceled cells in 
the columns which have no tilings up to this point 
are $b_1,b_2, \ldots, b_{k-1}$, 
We also cancel the same number 
of cells at the bottom of the corresponding columns of $\overline{B}$
\item[ii.] We can place a rook in any row of column $j$ that 
lies between the upper bar and lower bar.  This rook 
will not cancel anything. 
\item[iii.] We can put a flip of Fibonacci tiling in the $b_i$ 
uncanceled cells in column $b_j$ of $\overline{B}$.  This tiling will not 
cancel anything \end{itemize}

We let $\mathcal{M}_n(AugB_x)$ denote set of all 
mixed rook placements on $AugB_x$. For any $P \in \mathcal{M}_n(AugB_x)$, 
we let $\mathrm{one}_B(P)$ denote the number of tiles of height 1 that appear 
in $P$ that lie in $B$, 
$\mathrm{two}_B(P)$ denote the number of tiles of height 2 that appear 
in $P$ that lie in $B$, 
$\mathrm{one}_{\overline{B}}(P)$ denote the number of tiles 
of height 1 that appear 
in $P$ that lie in $\overline{B}$, 
and $\mathrm{two}_{\overline{B}}(P)$ denote the number of tiles 
of height 2 that appear 
in $P$ that lie in $\overline{B}$.  
We then define the weight of $P$, $\overline{WF}(P,p,q)$ to be 
$q^{\mathrm{one}_B(P)}p^{\mathrm{two}_B(P)} - 
q^{\mathrm{one}_{\overline{B}}(P)}p^{\mathrm{two}_{\overline{B}}(P)}$.
For example, 
Figure \ref{fig:aug2} pictures a mixed placement $P$ in 
$AugB_x$ where $B = F(2,3,4,4,5,5)$ and $x$ is 7 such that 
$\overline{WF}(P,p,q) =  q^3p^2 -q^2$. In this case 
we have put 2s in the cells that are canceled by the tiling 
in $B$ in column 2 and 4s in the cells that are canceled by 
the tiling in $B$ in column 4. Note that if 
we process the columns from left to right, 
after we have placed the tiling in column 2, 
the number of uncanceled cells in the columns which 
do not have tiling above the upper bar are 2,3,4,4,5 as 
we read from left to right in both $B$ and $\overline{B}$ 
and after we have placed the tiling in column 4, 
the number of uncanceled cells in the columns which 
do not have tilings  above the upper bar are 2,3,4,4 as 
we read from left to right in both $B$ and $\overline{B}$.

\fig{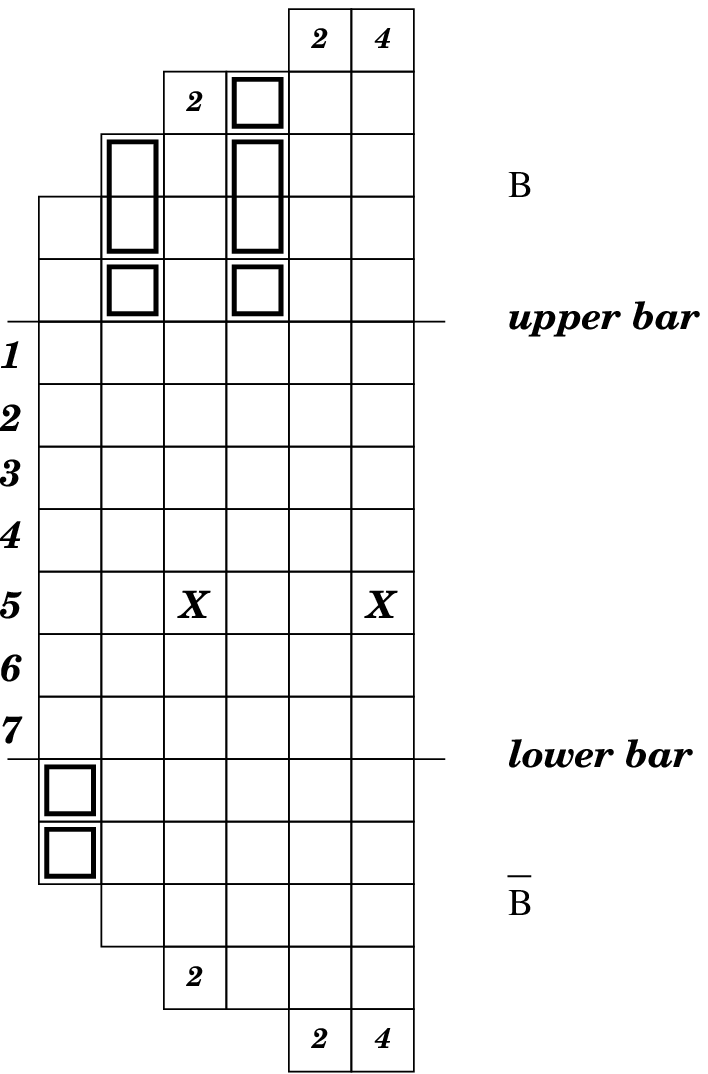}{A mixed rook placement.}

Our next theorem results from counting 
$\sum_{P \in \mathcal{M}_n(AugB_x)} \overline{WF}(P,p,q)$ 
in two different ways. 

\begin{theorem}\label{thm:Ferrersrookprod} 
Let $B =F(b_1, \ldots, b_n)$ be a Ferrers 
board where $0 \leq b_1 \leq \cdots \leq b_n$. 
\begin{equation}\label{Fibrookprod}
x^n =
\sum_{k=0}^n \mathbf{rT}_{n-k}(B,p,q) (x-F_{b_1}(p,q))(x-F_{b_2}(p,q))\cdots 
(x-F_{b_{k}}(p,q)).
\end{equation}
\end{theorem}
\begin{proof}
Since both sides of (\ref{Fibrookprod}) are polynomials of 
degree $n$, it is enough to show that there are 
$n+1$ different values of $x$ for which the two sides are equal. 
In fact, we will show that the two sides are equal for 
any positive integer $x$. 

Thus, fix $x$ to be a positive integer and consider 
the sum $S=\sum_{P \in \mathcal{M}_n(AugB_x)} \overline{WF}(P,p,q)$. 
First we consider the contribution of each column as we proceed 
from left to right. Given our three choices 
in column 1, the contribution of our choices of the tilings of 
height $b_1$ that we can place 
in column 1 of $B$ is $F_{b_1}(p,q)$, the contribution 
of our choices 
of placing a rook in between the upper  bar and the lower is $x$, 
and the contribution of our choices of the tilings of 
height $b_1$ that we can place 
in column 1 of $\overline{B}$ is $-F_{b_1}(p,q)$. 
Thus the contribution of our choices in 
column 1 to $S$ is $F_{b_1}(p,q)+x -F_{b_1}(p,q) = x$. 

In general, after we have processed our choices in 
the first $j$ columns, our cancellation scheme ensures 
that the number of uncanceled cells in $B$ and $\overline{B}$ in 
the $j$-th column is $b_i$ for some $i \leq j$. 
Thus given our three choices 
in column j, the contribution of our choices of the tilings of 
height $b_i$ that we can place in column $j$ of $B$ is $F_{b_i}(p,q)$, 
our choices 
of placing a rook in between the upper  bar and the lower is $x$, 
and the contribution of our choices of the tilings of 
height $b_i$ that we can place 
in column $j$ of $\overline{B}$ is $-F_{b_i}(p,q)$. 
Thus the contribution of our choices in 
column $j$ to $S$ is $F_{b_i}(p,q)+x -F_{b_i}(p,q) = x$.
It follows that $S = x^n$.

On the other 
hand, suppose that we fix a Fibonacci rook placement  
$P \in \mathcal{NT}_{n-k}(B)$. 
Then we want to compute the $S_P = \sum_{Q \in \mathcal{M}_n(AugB_x), 
Q \cap B = P} \overline{WF}(Q,p,q)$ which is the sum 
of $\overline{WF}(Q,p,q)$ over all 
mixed placements $Q$ such that $Q$ intersect $B$ equals $P$. 
Our cancellation scheme ensures that the number of uncanceled cells in $B$ and $\overline{B}$ 
in the $k$ columns that do not contain tilings in $P$ is 
$b_1, \ldots, b_k$ as we read from right to left. 
For each such $1 \leq i \leq k$, the factor that 
arises from either choosing a rook to be placed 
in between the upper bar and lower bar or a flipped 
Fibonacci tiling of height $b_i$ in $\overline{B}$ is 
$x-F_{b_i}(p,q)$. It follows that 
$$S_P =WF(P,p,q) \prod_{i=1}^k (x-F_{b_i}(p,q)).$$
Hence it follows that 
\begin{eqnarray*}
S &=&  \sum_{k=0}^n \sum_{P \in \mathcal{NT}_{n-k}(B)} S_P \\
&=& \sum_{k=0}^n \left( \prod_{i=1}^k (x-F_{b_i}(p,q))\right)  
\sum_{P \in \mathcal{NT}_k(B)} WF(P,p,q) \\
&=& \sum_{k=0}^n \mathbf{rT}_{n-k}(B,p,q) \left( \prod_{i=1}^k (x-F_{b_i}(p,q))\right).
\end{eqnarray*}
\end{proof}

 Now consider the special case of the previous two theorems 
when $B_n = F(0,1,2, \ldots, n-1)$. Then (\ref{Fibrookrec}) implies 
that 
$$\mathbf{rT}_{n+1-k}(B_{n+1},p,q) = \mathbf{rT}_{n+1-k}(B_n,p,q) + F_k(p,q) 
\mathbf{rT}_{n-k}(B_n,p,q).$$
It then easily follows that for all $0 \leq k \leq n$, 
\begin{equation}\label{Sfnkpqinterp}
 \mathbf{Sf}_{n,k}(p,q) = \mathbf{rT}_{n-k}(B_n,p,q).
\end{equation}
Note that $\mathbf{Sf}_{n,0}(p,q) = 0$ for all $n \geq 1$ since 
there are no Fibonacci rook placements in 
$\mathcal{NT}_n(B_n)$ since there are only $n-1$ non-zero columns. 
Moreover such a situation, we see that (\ref{Fibrookprod}) 
implies that 
\begin{equation*}
 x^n = 
\sum_{k=1}^n \mathbf{Sf}_{n,k}(p,q) x(x-F_1(p,q))(x-F_2(p,q))\cdots (x-F_{k-1}(p,q))
\end{equation*}
Thus we have given a combinatorial proof of 
(\ref{pqSfdef}).

\section{A combinatorial proof that $||\mathbf{Sf}_{n,k}(p,q)||^{-1} 
= ||\mathbf{sf}_{n,k}(p,q)||$.}

In this section, we shall give a combinatorial proof 
that the infinite matrices $||\mathbf{Sf}_{n,k}(p,q)||_{n,k \geq 0}$ 
and $||\mathbf{sf}_{n,k}(p,q)||_{n,k \geq 0}$ are inverses of each other.

Since the matrices  $||\mathbf{Sf}_{n,k}(p,q)||_{n,k \geq 0}$ 
and $||\mathbf{sf}_{n,k}(p,q)||_{n,k \geq 0}$ are lower triangular, 
we must prove that for any $0 \leq k \leq n$, 
\begin{equation}\label{inverse1}
\sum_{j =k}^n \mathbf{Sf}_{n,j}(p,q)\mathbf{sf}_{j,k}(p,q) = \chi(n =k)
\end{equation}
where, for a statement $A$, we let $\chi(A) =1$ if $A$ is true 
and $\chi(A) =0$ if $A$ is false. Given our combinatorial 
interpretation of $\mathbf{Sf}_{n,j}(p,q)$ and $\mathbf{sf}_{j,k}(p,q)$,  
we must show that 
\begin{equation}\label{inverse2}
 \sum_{j =k}^n \sum_{(P,Q) \in 
\mathcal{NT}_{n-j}(B_n) \times \mathcal{FT}_{j-k}(B_j)}
(-1)^{j-k}WF(P,p,q) WF(Q,p,q) = \chi(n =k).
\end{equation}
Note that if $(P,Q) \in 
\mathcal{NT}_{n-j}(B_n) \times \mathcal{FT}_{j-k}(B_j)$, 
then the sign associated with $(P,Q)$ is just 
$(-1)^{\mbox{no. of rooks in $Q$}}$ so that 
we define $sgn(P,Q) = (-1)^{\mbox{no. of rooks in $Q$}}$.

In the case when $n=k$, (\ref{inverse2}) 
reduces to the fact that 
$$1=\sum_{(P,Q) \in 
\mathcal{NT}_{n-n}(B_n) \times \mathcal{FT}_{n-n}(B_n)}
(-1)^{n-n}WF(P,p,q) WF(Q,p,q).$$
This is clear since the only $P$ in $\mathcal{NT}_{n-n}(B_n)$ is the empty 
configuration and $WF(P,p,q) =1$ and the only $Q$ in 
$\mathcal{FT}_{n-n}(B_n)$ is the empty 
configuration and $WF(Q,p,q) =1$. Moreover in such a case 
$sgn(P,Q) =1$. 

If $n \geq 1$ and $k=0$, the result is also immediate. In 
that case our sum becomes 
$$\sum_{j=0}^n \sum_{(P,Q) \in 
\mathcal{NT}_{n-j}(B_n) \times \mathcal{FT}_{j-0}(B_j)}
(-1)^{n-n}WF(P,p,q) WF(Q,p,q).$$
However, for all $j \geq 1$, $\mathcal{FT}_{j}(B_j)$ is empty 
because you can place at most $j-1$ file tilings on $B_j$.  
In the case when $j=0$, then $\mathcal{NT}_{n-0}(B_n)$ is empty so 
that the entire sum is empty. Hence for all $n \geq 1$, 
$$\sum_{j =0}^n \mathbf{Sf}_{n,j}(p,q)\mathbf{sf}_{j,0}(p,q) =0.$$

Thus we can assume that $n > k \geq 1$. 
Our goal is to define an involution 
$$I_{n,k}: \bigcup_{j=k}^n  
\mathcal{NT}_{n-j}(B_n) \times \mathcal{FT}_{j-k}(B_j) \rightarrow 
\bigcup_{j=k}^n \mathcal{NT}_{n-j}(B_n) \times \mathcal{FT}_{j-k}(B_j)$$
such that for all for
$(P,Q)\in \bigcup_{j=k}^n \mathcal{NT}_{n-j}(B_n) \times \mathcal{FT}_{j-k}(B_j)$, $I_{n,k}(P,Q) = (P',Q') \neq (P,Q)$, $sgn(P,Q) = -sgn(P',Q')$, 
and $WF(P,p,q)WF(Q,p,q) = WF(P',p,q)WF(Q',p,q)$.

We shall proceed 
by induction on $n$. The base case of 
our induction is $n =2$ and $k=1$. In that case (\ref{inverse2})  
becomes 
$$
\sum_{j =1}^2 \sum_{(P,Q) \in 
\mathcal{NT}_{2-j}(B_2) \times \mathcal{FT}_{j-1}(B_j)}
(-1)^{j-1}WF(P,p,q) WF(Q,p,q).$$
However, in the case $j=2$, there is a single pair in 
$\mathcal{NT}_{2-1}(B_2) \times \mathcal{FT}_{1-1}(B_1)$ which 
is pictured on the left in Figure \ref{fig:base} and 
there is a single pair in 
$\mathcal{NT}_{2-2}(B_2) \times \mathcal{FT}_{2-1}(B_2)$ which 
is pictured on the right in Figure \ref{fig:base}.  These 
two pairs each have weight $q$ but have opposite signs 
so that our involution $I_{2,1}$ just 
maps each pair to the other pair. 

\fig{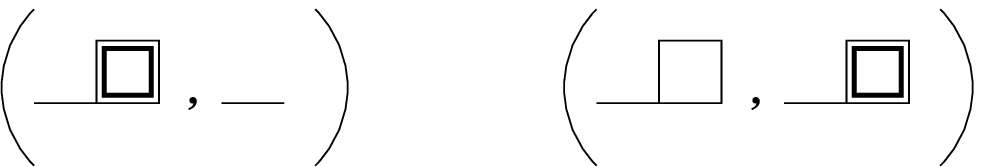}{$I_{2,1}$.}

Thus assume that $n > 2$ and $n >k \geq 1$. We define 
$I_{n,k}$ via 3 cases. \\
\ \\
{\bf Case 1.} $(P,Q) \in \mathcal{NT}_{n-j}(B_n) 
\times \mathcal{FT}_{j-k}(B_j)$ and there is a tiling in 
the last column of $P$.\\

In this case, there are $n-j-1$ tilings in the first 
$n-1$ columns of $P$ so that there are $n-1-(n-j-1) =j$ 
uncanceled cells in the last column of $P$. Note that the 
last column of $Q$ is of height $j-1$. Then 
we let $I_{n,k}(P,Q) = (P',Q')$ where 
$P'$ arises from $P$ by removing the tiling in the last 
column of $P$ and $Q'$ results from $Q$ by taking 
the tiling in the last column of $P$ and placing it at 
the end of $Q$.  It will then be the case 
that $(P',Q') \in \mathcal{NT}_{n-(j+1)}(B_n) 
\times \mathcal{FT}_{j+1-k}(B_{j+1})$. Note that 
$sgn(P,Q) =(-1)^{j-k}$ and $sgn(P',Q') =(-1)^{j+1-k}$. 
In addition, since we did not change the total number of 
tiles of size 1 and 2, we have that 
$WF(P,p,q)WF(Q,p,q) = WF(P',p,q)WF(Q',p,q)$.\\
\ \\
{\bf Case 2.} $(P,Q) \in \mathcal{NT}_{n-j}(B_n) 
\times \mathcal{FT}_{j-k}(B_j)$ and there is no tiling in 
the last column of $P$ but there is a tiling in the 
last column of $Q$.\\

In this case, there are $n-j$ tilings in the first 
$n-1$ columns of $P$ so that there are $n-1-(n-j) =j-1$ 
uncanceled cells in the last column of $P$. Note that the 
last column of $Q$ is of height $j-1$ in this case. 
Then 
we let $I_{n,k}(P,Q) = (P',Q')$ where 
$P'$ arises from $P$ by taking the tiling 
in the last column of $Q$ and placing it in the $j-1$ uncanceled 
cells of the last column of $P$ and $Q'$ results from $Q$ removing 
its last column. 

It will then be the case 
that $(P',Q') \in \mathcal{NT}_{n-(j-1)}(B_n) 
\times \mathcal{FT}_{j-1-k}(B_{j-1})$. Note that 
$sgn(P,Q) =(-1)^{j-k}$ and $sgn(P',Q') =(-1)^{j-1-k}$. 
Since we did not change the total number of 
tiles of size 1 and 2, we have that 
$WF(P,p,q)WF(Q,p,q) = WF(P',p,q)WF(Q',p,q)$.\\

It is easy to see that if $(P,Q)$ is in Case 1, then 
$I_{n,k}(P,Q)$ is in Case 2 and if $(P,Q)$ is in Case 2, 
then $I_{n,k}(P,Q)$ is in Case 1.  An example of 
these two cases is given in Figure \ref{fig:CASEI} 
where the pair $(P,Q)$ pictured at the top is in 
$\mathcal{NT}_{6-3}(B_6) 
\times \mathcal{FT}_{3-2}(B_3)$ and satisfies the 
conditions of Case 1 and the pair $(P',Q')$ pictured at the bottom is in 
$\mathcal{NT}_{6-4}(B_6) 
\times \mathcal{FT}_{4-2}(B_4)$ and satisfies the 
conditions of Case 2. 

\fig{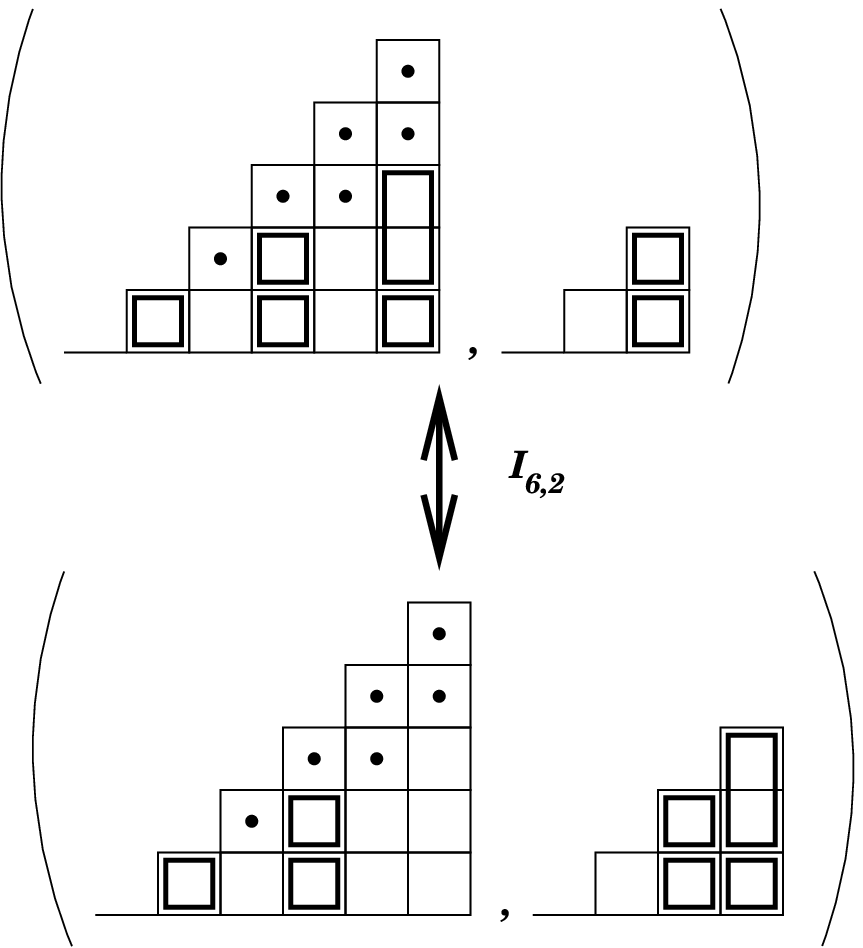}{The involution $I_{n,k}$.}

It follows that sum of $sgn(P,Q)WF(P,p,q)WF(Q,p,q)$ over all 
$(P,Q)$ satisfying the conditions of Case 1 or Case 2 for some $j$ 
is 0.  Thus we have one last case to consider. \\
\ \\
{\bf Case 3.} $(P,Q) \in \mathcal{NT}_{n-j}(B_n) 
\times \mathcal{FT}_{j-k}(B_j)$ and there is no tiling in 
the last column of $P$ and there is no tiling in the 
last column of $Q$.\\

In this case, we let $(P',Q')$ be the result  of 
removing the last column of both $P$ and $Q$. It follows 
that $(P',Q') \in \mathcal{NT}_{n-1-(j-1)}(B_{n-1}) 
\times \mathcal{FT}_{(j-1)-(k-1)}(B_{j-1})$. 
It is easy to see that the map $\theta$ which sends 
$(P,Q) \rightarrow (P',Q')$ 
is a sign preserving and weight preserving bijection  
from the set of all 
$(P,Q) \in \bigcup_{j=k}^n \mathcal{NT}_{n-j}(B_n) 
\times \mathcal{FT}_{j-k}(B_j)$ which are in case 3 onto 
the set $\bigcup_{i=k-1}^{n-1} \mathcal{NT}_{n-1-i}(B_{n-1}) 
\times \mathcal{FT}_{i-(k-1)}(B_i)$. 
An example of the $\theta$ maps is given in Figure 
\ref{fig:theta}. 
But then we know 
by induction that 
$$\sum_{(P',Q') \in \bigcup_{i=k-1}^{n-1} \mathcal{NT}_{n-1-i}(B_{n-1}) 
\times \mathcal{FT}_{i-(k-1)}(B_i)} sgn(P',Q') WF(P',p,q)WF(Q',p,q)
 =0.$$

This shows that (\ref{inverse2}) holds which is what we wanted to prove.

\fig{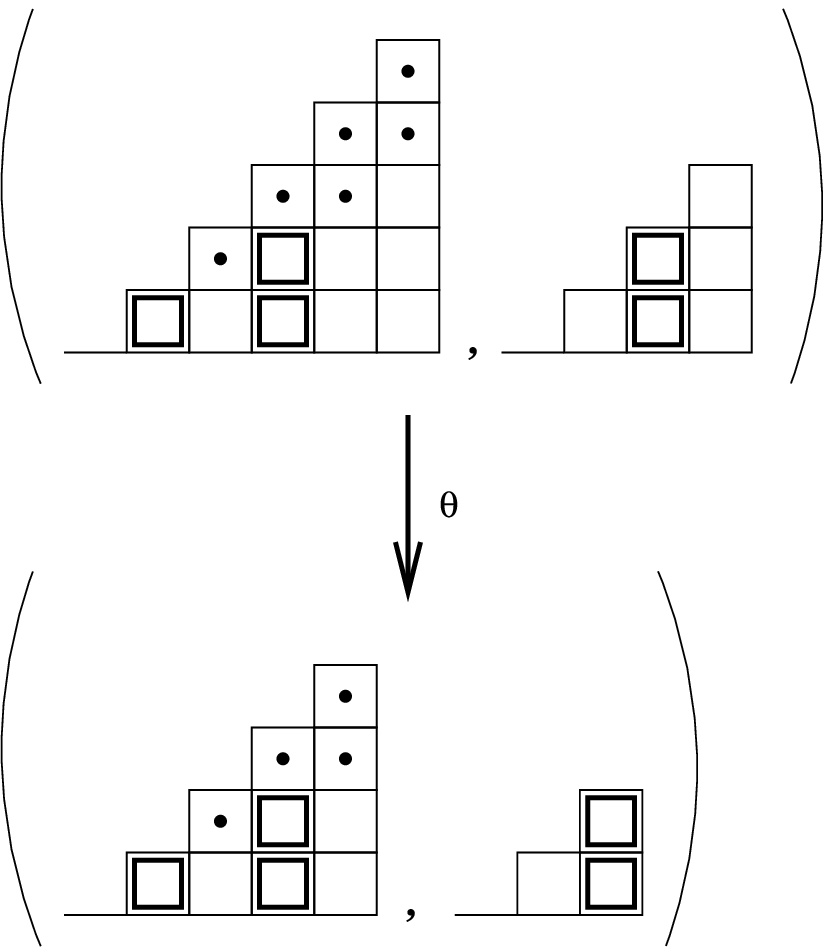}{An example of the $\theta$ map.}

\noindent {\bf Remark.}  It is to see that the 
proofs in Sections 2 and 3 did not use any particular  
properties of the Fibonacci tilings. Thus for example, 
let $\mathcal{PT}_n$ denote the set of all tilings of a column 
of height $n$ with tiles of height 1, 2, or 3 such that 
the bottom tile is a size 1.

we could consider numbers defined 
$P_1=P_2 =1$, $P_3 =2$ and $P_n = P_{n-1}+P_{n-2}+P_{n-3}$ for $n \geq 4$. 
Then it then easy to see that $P_n$ equals the number of 
tilings of height $n$ using tiles of size 1, 2, and 3 such 
that the bottom tile is of size 1.  We will call such tilings 
$P$-tilings. Given tiling $T \in \mathcal{PT}_n$, we let 
$\mathrm{one}(T)$ denote the number of tiles of height 1 in $T$, 
$\mathrm{two}(T)$ denote the number of tiles of height 2 in $T$, and
$\mathrm{three}(T)$ denote the number of tiles of height 3 in $T$. 
Then we let 
$$P_n(p,q,r) = \sum_{T \in \mathcal{PT}_n}q^{\mathrm{one}(T)}p^{\mathrm{two}(T)} r^{\mathrm{three}(T)}.$$

For example, 
Figure \ref{fig:Ptilings} gives the set of tiles for 
$\mathcal{PT}_1, \ldots, \mathcal{PT}_5$. 

\fig{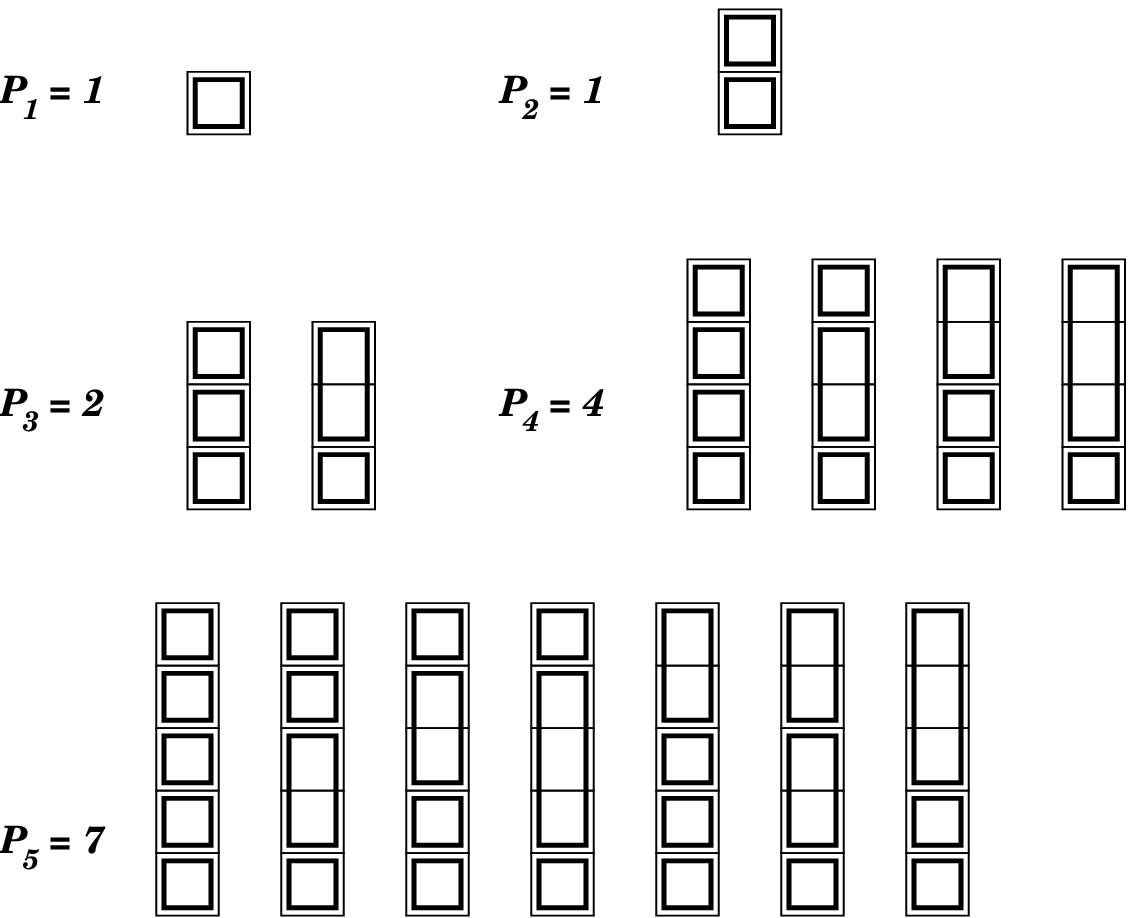}{$P$-tilings for $\mathcal{PT}_1, \ldots, \mathcal{PT}_5$.}

Given any such $P$ tiling $T$, we let 
$\mathrm{one}(T)$ denote the number of tiles of height 1 in $T$, 
$\mathrm{two}(T)$ denote the number of tiles of height 2 in $T$, and
$\mathrm{three}(T)$ denote the number of tiles of height 3 in $T$. 
The we can define $P_n(q,p,r)$ as the sum of weights $WP(T)$ 
of all $P$-tilings $T$ of height $n$ where 
$WP(T) =q^{\mathrm{one}(T)} p^{\mathrm{two}(T)} r^{\mathrm{three}(T)}$ 
over all tilings for $P_n$.  It is then easy to see that 
$P_1(q,p,r)=q$, $P_2(q,p,r) =q^2$, $P_3(q,p,r)=q^3+qp$, and 
$$P_n(q,p,r)= q P_{n-1}(q,p,r)+pP_{n-2}(q,p,r)+rP_{n-3}(q,p,r)$$
for $n \geq 4$.

Then for any Ferrers board $B$, 
we can define $P$-analogues $\mathbf{rPT}_{k}(B,p,q,r)$
of the rook numbers $\mathbf{rT}_{k}(B,p,q)$ and $P$-analogues 
$\mathbf{fPT}_{k}(B,p,q,r)$ of the file numbers 
 $\mathbf{fT}_{k}(B,p,q)$ exactly as before except that 
we replace Fibonacci tilings by $P$-tilings and we keep 
track of the number of tiles of size 1, 2, and 3 instead 
of keeping track of the tiles of size 1 and 2. 
Then we have the following analogues of 
Theorems \ref{thm:Ferrersfilerec}, \ref{thm:Ferrersfileprod}, 
\ref{thm:Ferrersrookrec}, and \ref{thm:Ferrersrookprod} 
with basically the same proofs.  

\begin{theorem}\label{thm:PFerrersfilerec} 
Let $B =F(b_1, \ldots, b_n)$ be a Ferrers 
board where $0 \leq b_1 \leq \cdots \leq b_n$ and $b_n > 0$. 
Let $B^- = F(b_1, \ldots, b_{n-1})$. Then for all 
$1 \leq k \leq n$, 
\begin{equation}\label{PFibtilerec}
\mathbf{fPT}_k(B,p,q,r) = \mathbf{fPT}_{k}(B^-,p,q,r)+ P_{b_n}(p,q,r) \mathbf{fPT}_{k-1}(B^-,p,q,r).
\end{equation}
\end{theorem}

\begin{theorem}\label{thm:PFerrersfileprod}
Let $B =F(b_1, \ldots, b_n)$ be a Ferrers 
board where $0 \leq b_1 \leq \cdots \leq b_n$. 
\begin{equation}\label{PFibfileprod}
(x+P_{b_1}(p,q,r))(x+P_{b_2}(p,q,r)) \cdots (x+P_{b_n}(p,q,r)) =
\sum_{k=0}^n \mathbf{fPT}_k(B,p,q,r) x^{n-k}.
\end{equation}
\end{theorem}

\begin{theorem}\label{thm:PFerrersrookrec} 
Let $B =F(b_1, \ldots, b_n)$ be a Ferrers 
board where $0 \leq b_1 \leq \cdots \leq b_n$ and $b_n > 0$. 
Let $B^- = F(b_1, \ldots, b_{n-1})$. Then for all 
$1 \leq k \leq n$, 
\begin{equation}\label{PFibrookrec}
\mathbf{rPT}_k(B,p,q,r) = \mathbf{rPT}_{k}(B^-,p,q,r)+ P_{b_{n-(k-1)}}(p,q,r) 
\mathbf{rPT}_{k-1}(B^-,p,q,r).
\end{equation}
\end{theorem}

\begin{theorem}\label{thm:PFerrersrookprod} 
Let $B =F(b_1, \ldots, b_n)$ be a Ferrers 
board where $0 \leq b_1 \leq \cdots \leq b_n$. 
\begin{equation}\label{PFibrookprod}
x^n =
\sum_{k=0}^n \mathbf{rPT}_{n-k}(B,p,q,r) 
(x-P_{b_1}(p,q,r))(x-P_{b_2}(p,q,r))\cdots 
(x-P_{b_{k}}(p,q,r)).
\end{equation}
\end{theorem}

In particular, if we let 
$\mathbf{cp}_{n,k}(p,q,r) = \mathbf{fPT}_{n-k}(B_n,p,q,r)$,  
then we will have that \\
$\mathbf{cp}_{n,n}(p,q,r) =1$ for 
all $n \geq 0$, $\mathbf{cp}_{n,0}(p,q,r) =0$ for all $n \geq 1$, 
and 
$$\mathbf{cp}_{n+1,k}(p,q,r) =\mathbf{cp}_{n,k-1}(p,q,r) + 
P_n(p,q,r) \mathbf{cp}_{n,k}(p,q,r)$$
for $1 \leq k \leq n+1$ and 
$$x(x+P_1(p,q,r)) \cdots (x+P_{n-1}(p,q,r)) = \
\sum_{k=1}^n \mathbf{cp}_{n,k}(p,q,r)\ x^k.$$

Similarly, if we let 
$\mathbf{Sp}_{n,k}(p,q,r) = \mathbf{rPT}_{n-k}(B_n,p,q,r)$, 
then we will have that $\mathbf{Sp}_{n,n}(p,q,r) =1$ for 
all $n \geq 0$, $\mathbf{Sp}_{n,0}(p,q,r) =0$ for all $n \geq 1$, 
and 
$$\mathbf{Sp}_{n+1,k}(p,q,r) =\mathbf{Sp}_{n,k-1}(p,q,r) + 
P_k(p,q,r) \mathbf{Sp}_{n,k}(p,q,r)$$
for $1 \leq k \leq n+1$ and 
$$x^n = \sum_{k=1}^n \mathbf{Sp}_{n,k}(p,q,r)
x(x-P_1(p,q,r)) \cdots (x-P_{k-1}(p,q,r)).$$

Finally, if we let $\mathbf{sp}_{n,k}(p,q,r) = (-1)^{n-k} 
\mathbf{cp}_{n,k}(p,q,r)$, then essentially the same proof 
that we used in this section, we will give a combinatorial 
proof of the fact that the matrices $||\mathbf{Sp}_{n,k}(p,q)||$ and  
$||\mathbf{sp}_{n,k}(p,q)||$ are inverses of each other.

\section{Identities for $\mathbf{Sf}_{n,k}(p,q)$ and 
$\mathbf{cf}_{n,k}(p,q)$}

In this section, we shall derive various identities for the 
Fibonacci analogues of the Stirling numbers $\mathbf{Sf}_{n,k}(p,q)$ and 
$\mathbf{cf}_{n,k}(p,q)$. We let  $[0]_q =1$ and, for any positive integer $n$, let $[n]_q =1+q+ \cdots + q^{n-1}$. Then the usual $q$-analogues of $n!$ and $\binom{n}{k}$ are defined by 
\begin{eqnarray*}
[n]_q! &=& [n]_q[n-1]_q \cdots [2]_q [1]_q \ \mbox{and} \\
\qbinom{n}{k} &=& \frac{[n]_q!}{[k]_q![n-k]_q!}.
\end{eqnarray*}

A partition $\lambda =(\lambda_1, \ldots, \lambda_k)$ of $n$ 
is a weakly increasing sequence of positive integers such that 
$\sum_{i=1}^k \lambda_i =n$. We let $|\lambda|=n$ denote 
the size of $\lambda$ and $\ell(\lambda) =k$ denote the number 
of parts of $\lambda$. For this paper, we will draw 
the Ferrers diagram of a partition consistent 
with the convention for Ferrers boards. That is, 
the Ferrers diagram of $\lambda = (\lambda_1 \leq \cdots \leq \lambda_k)$ 
is the Ferrers board $F(\lambda_1, \ldots, \lambda_k)$. 
A standard combinatorial interpretation of the $q$-binomial coefficient 
$\qbinom{n}{k}$ is that $\qbinom{n}{k}$
equals the sum of $q^{|\lambda|}$ over all 
partitions whose Ferrers diagram are contained in the 
$(n-k) \times k$ rectangle.

We have already seen that 
$\mathbf{cf}_{n,n}(p,q) = \mathbf{Sf}_{n,n}(p,q) =1$ since 
these correspond to the empty placements in $B_n$. 
Then we have the following simple theorem.

\begin{theorem}\label{thm:id1} 

For all $n \geq 1$, 
\begin{equation*}
\mathbf{Sf}_{n,1}(p,q) =  q^{n-1} \ \mbox{and} \ 
\mathbf{cf}_{n,1}(p,q) =  \prod_{i=1}^{n-1} F_i(p,q).
\end{equation*} 

For all $n \geq 2$, 
\begin{eqnarray*}
\mathbf{Sf}_{n,2}(p,q) &=& q^{n-2}[n-1]_q \ \mbox{and}  \\ 
\mathbf{cf}_{n,n-1}(p,q) &=& \mathbf{Sf}_{n,n-1}(p,q) =
\sum_{i=1}^{n-1} F_i(p,q).
\end{eqnarray*} 
\end{theorem}
\begin{proof}

For $\mathbf{Sf}_{n,1}(p,q)$, we must count  
the weights of the Fibonacci 
rook tilings of $B_n$ in which every column has a tiling. 
For each column $i \geq 2$ in $B_n$, our cancellation scheme 
ensures that all but one square in column $i$ is canceled by 
the tilings to its left.  Thus there is only one Fibonacci 
rook tiling that contributes to $\mathbf{Sf}_{n,1}(p,q)$ which 
is the tiling where every column has exactly one tile of height 
1. For example, Figure \ref{fig:special5} picture such a tiling 
for  $\mathbf{Sf}_{5,1}(p,q)$. Hence 
$\mathbf{Sf}_{n,1}(p,q) =q^{n-1}$. 

\fig{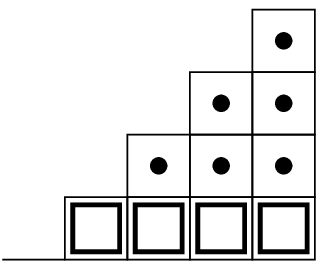}{The tiling for $\mathbf{Sf}_{5,1}(p,q)$.}

For $\mathbf{cf}_{n,1}(p,q)$, we must count the weights of the Fibonacci file tilings of $B_n$ in which every column has a tiling. 
Since the sum of the weights of the tilings in column 
$i$ is $F_{i-1}(p,q)$ for $i =2, \ldots, n$, it follows 
that $\mathbf{cf}_{n,1}(p,q)= \prod_{i=1}^{n-1} F_i(p,q)$.

For $\mathbf{Sf}_{n,2}(p,q)$, we know that $\mathbf{Sf}_{2,2}(p,q) =1$
so that our formula holds for $n =2$.   For $n \geq 3$, we 
must count the weights of all the Fibonacci rook tilings such that 
there is exactly one empty column. It is easy to see 
that if the empty column is at the end, then by our argument 
of $\mathbf{Sf}_{n,1}(p,q)$, there is exactly one tile in each 
columns $2, \ldots, n-1$ so that the weight of such a tiling 
is $q^{n-2}$.  Then as the empty column moves right to left, 
we see that we replace a column with one tile with a column 
with two tiles. This process is pictured in Figure \ref{fig:Sfn2} 
for $n =6$. It follows that for $n \geq 3$, 
$$\mathbf{Sf}_{n,2}(p,q) = q^{n-2}+q^{n-1} \cdots + q^{2(n-2)} 
=q^{n-2}(1+q+ \cdots + q^{n-2}) = q^{n-2}[n-1]_q.$$

\fig{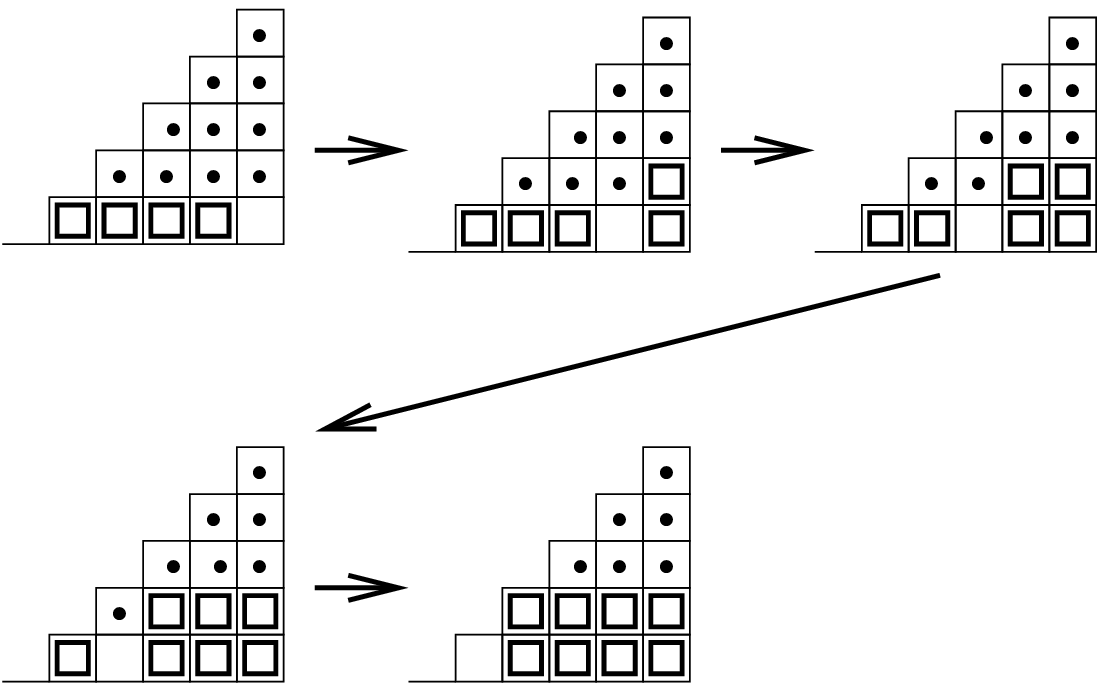}{The tilings for $\mathbf{Sf}_{6,2}(p,q)$.}

For $\mathbf{Sf}_{n,n-1}(p,q)$ and $\mathbf{cf}_{n,n-1}(p,q)$,  
we must count the tilings of $B_n$ in which exactly one 
column is tiled. In this case, the  rook tilings and the file tilings are 
the same. Hence $\mathbf{cf}_{n,n-1}(p,q) = \mathbf{Sf}_{n,n-1}(p,q) =
\sum_{i=1}^{n-1} F_i(p,q)$.
\end{proof}

Next we define  
$$\mathbb{SF}_k(p,q,t) := \sum_{n \geq k} \mathbf{Sf}_{n,k}(p,q) t^n$$ 
for $k \geq 1$
It follows from Theorem \ref{thm:id1} that 
\begin{equation}\label{eq:id2}
\mathbb{SF}_1(p,q,t) = 
\sum_{n \geq 1} q^{n-1}t^n = \frac{t}{(1-qt)}=  \frac{t}{(1-F_1(p,q)t)}.
\end{equation}
Then for $k > 1$,  
\begin{eqnarray*} 
\mathbb{SF}_k(p,q,t) &=& \sum_{n \geq k} \mathbf{Sf}_{n,k}(p,q) t^n \\
&=& t^k + \sum_{n > k} \mathbf{Sf}_{n,k}(p,q) t^n \\
&=& t^k + t \sum_{n > k} \left( \mathbf{Sf}_{n-1,k-1}(p,q) +
F_k(p,q) \mathbf{Sf}_{n-1,k-1}(p,q)\right) t^{n-1} \\
&=& t^k + t\left(\sum_{n > k}  \mathbf{Sf}_{n-1,k-1}(p,q) t^{n-1}\right) 
+ F_k(p,q)t\left(\sum_{n > k}  \mathbf{Sf}_{n-1,k}(p,q) t^{n-1}\right) \\
&=& t^k + t(\mathbb{SF}_{k-1}(p,q,t) -t^{k-1}) +  F_k(p,q)t \mathbb{SF}_k(p,q,t).
\end{eqnarray*}
It follows 
that 
\begin{equation}\label{eq:id3}
\mathbb{SF}_k(p,q,t) = \frac{t}{(1 - F_k(p,q)t)}\mathbb{SF}_{k-1}(p,q,t).
\end{equation}
The following theorem easily follows from (\ref{eq:id2}) and 
(\ref{eq:id3}).

\begin{theorem}\label{thm:id4} For all $k \geq 1$, 
$$\mathbb{SF}_k(p,q,t)= \frac{t^k}{(1-F_1(p,q)t) (1-F_2(p,q)t)\cdots 
(1-F_k(p,q)t)}.$$
\end{theorem}

For any formal power series in $f(x) = \sum_{n \geq 0}f_n x^n$, 
we let $f(x)|_{x^n} = f_n$ denote the coefficient of $x^n$ in 
$f(x)$. 

Our next result will give formulas for 
$\mathbf{Sf}_{n,k}(p,q)|_{p^0}$ and $\mathbf{Sf}_{n,k}(p,q)|_{p}$. 
Note that we have already shown that 
$\mathbf{Sf}_{n,2}(p,q) =q^{n-2}[n-1]_q$ so that $\mathbf{Sf}_{n,2}(p,q)|_{p} =0$. 

\begin{theorem}\label{thm:id5} 

For all $n \geq k \geq 1$, 
\begin{equation}\label{eq:id5}
\mathbf{Sf}_{n,k}(p,q)|_{p^0} =q^{n-k}\qbinom{n-1}{k-1} 
\end{equation}
and for all $n > k \geq 3$, 
\begin{equation}\label{eq:id6}
\mathbf{Sf}_{n,k}(p,q)|_{p} =
q^{n-k} \sum_{s=1}^{k-2} sq^{s-1}\sum_{i=0}^{n-k-1} q^{i(s+1)}
\qbinom{i+k-s-2}{i} \qbinom{s+ n-k-i}{s+1}.
\end{equation}
\end{theorem}
\begin{proof} 

There are two proofs that we can give for (\ref{eq:id5}). 
The first uses the generating function
\begin{equation}\label{eq:id7}
\mathbb{SF}_k(p,q,t) = \sum_{n \geq k} \mathbf{Sf}_{n,k}(p,q) t^n 
=  \frac{t^k}{(1-F_1(p,q)t) (1-F_2(p,q)t)\cdots 
(1-F_k(p,q)t)}.
\end{equation}
Clearly, for all $n \geq 1$, $F_n(p,q)|_{p^0} =q^n$ because 
the only Fibonacci tiling $T \in \mathcal{FT}_n$ which has 
no tiles of height 2 is the tiling which consists of 
$n$ tiles of height 1. 
Taking the coefficient of $p^0$ on both sides of (\ref{eq:id7}), we see 
that 
\begin{eqnarray}\label{eq:id8}
\sum_{n \geq k} \mathbf{Sf}_{n,k}(p,q)|_{p^0}t^n 
&=&   \frac{t^k}{(1-F_1(p,q)|_{p^0}t) (1-F_2(p,q)|_{p^0}t)\cdots 
(1-F_k(p,q)|_{p^0}t)} \nonumber \\
&=& \frac{t^k}{(1-qt)(1-q^2t) \cdots (1-q^kt)}.
\end{eqnarray}

Taking the coefficient of 
$t^n$ on both sides of (\ref{eq:id8}), we see 
that $\mathbf{Sf}_{n,k}(p,q)|_{p^0}$ equals 
the sum of $q^{|\lambda|}$ over all partitions $\lambda$ with 
$n-k$ parts whose parts are from $\{1, \ldots, k\}$. If 
we subtract 1 from each part of $\lambda$, we will end up 
with a partition contained in the $(n-k) \times (k-1)$. Since 
the sum of $q^{|\pi|}$ over all partitions $\pi$ whose 
Ferrers diagram is contained in $(n-k) \times (k-1)$ rectangle 
is $\qbinom{n-1}{k-1}$, it follows that 
$\mathbf{Sf}_{n,k}(p,q)|_{p^0} = q^{n-k}\qbinom{n-1}{k-1}$.

In fact, this result can be seen directly from our rook 
theory interpretation for $\mathbf{Sf}_{n,k}(p,q)$.  That 
is, $\mathbf{Sf}_{n,k}(p,q)|_{p^0}$ is the sum of $q^{\mathrm{one}(T)}$ over 
all Fibonacci rook tilings $T$ of $B_n$ with $k$ empty columns 
that only use tiles of height $1$.  It is easy to see that 
that every time we traverse an empty column, the number 
of tiles that we can put in a column goes up by one. Since the 
first column is empty, this means that we can start with 
tilings of height 1 and as we traverse the $k-1$ remaining 
empty column, the maximum number of tiles of height one 
that we can put in any column is $k$. It follows 
that if we remove the tiles of height $1$ at the bottom 
of any Fibonacci rook tiling $T$ of $B_n$ with $k$ empty columns 
that use only tiles of height 1, we will be left with a Ferrers 
diagram of a partition which is contained in the $(n-k) \times (k-1)$ 
rectangle.  This process is pictured in Figure \ref{fig:Sfnkp0} 
in the case where $n =11$ and $k =4$. 

\fig{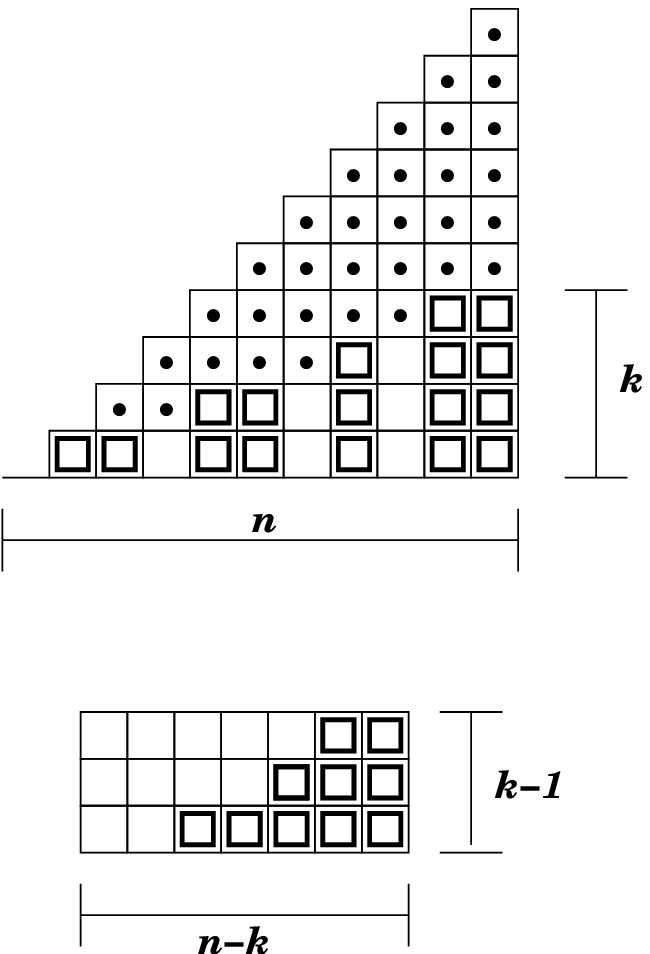}{The correspondence between tilings in 
$\mathcal{NT}_{n-k}(B_n)$ using only  
tiles of height 1 and partitions contained in the 
$(n-k) \times (k-1)$ rectangle.}

We can also reverse this correspondence.  That is, 
if we are given the Ferrers diagram of partition $\mu$ contained 
the $(n-k) \times (k-1)$ rectangle, we can reconstruct 
the tiling $P \in \mathcal{NT}_{n_k}(B_n)$ which gave rise 
to $\mu$.  That is, we first add tiles of height 1 at the bottom of $\mu$ which will give us 
the Ferrers diagram of partition $\lambda = (\lambda_1, 
\ldots, \lambda_{n-k})$ with $n-k$ parts 
with parts from $\{1, \ldots, k\}$. Thus 
$1 \leq \lambda_1 \leq \cdots \leq \lambda_{n-k} \leq k$. 
Then if $\lambda_1=1$, we put a tiling of height $1$ in 
column 2. If $\lambda_1 =j >1$, then we start 
with $j$ empty columns and place a tiling of height 
$j$ in column $j+1$.  Then assuming that we have 
placed the tilings corresponding to $\lambda_1, \ldots, 
\lambda_{i}$, we place the tiling for $\lambda_{i+1}$ 
next to the tiling for $\lambda_i$ if $\lambda_i = \lambda_{i+1}$ 
and we put $\lambda_{i+1} - \lambda_i$ consecutive empty columns next 
to the column that contains the tiling for $\lambda_i$ and 
place the tiling for $\lambda_{i+1}$ in the next column if 
$\lambda_{i+1} - \lambda_i > 0$. This process 
is pictured in Figure \ref{fig:Sfnkp02} in the case 
where $n =11$ and $k =5$.  That is, we start with 
the partition $(0,0,0,2,3,3)$ contained in the $6 \times 4$ rectangle. 
Then we add one to each part of the partition to get the 
partition $\lambda = (1,1,1,3,4,4)$.  Then our process says we 
start putting a tiling of height 1 in column 2 and add two 
more tiling of height 1 columns 3 and 4. Next, since 
$\lambda_4 - \lambda_3 =2$, we put two empty columns followed 
by a column of height $3$ in column 7.  Next, since 
$\lambda_5 - \lambda_4 =1$, we put an empty column followed 
by a column of height $4$ in column 9.  Finally 
since $\lambda_5 = \lambda_6$, we add another column of 
height $4$ in column 10.

\fig{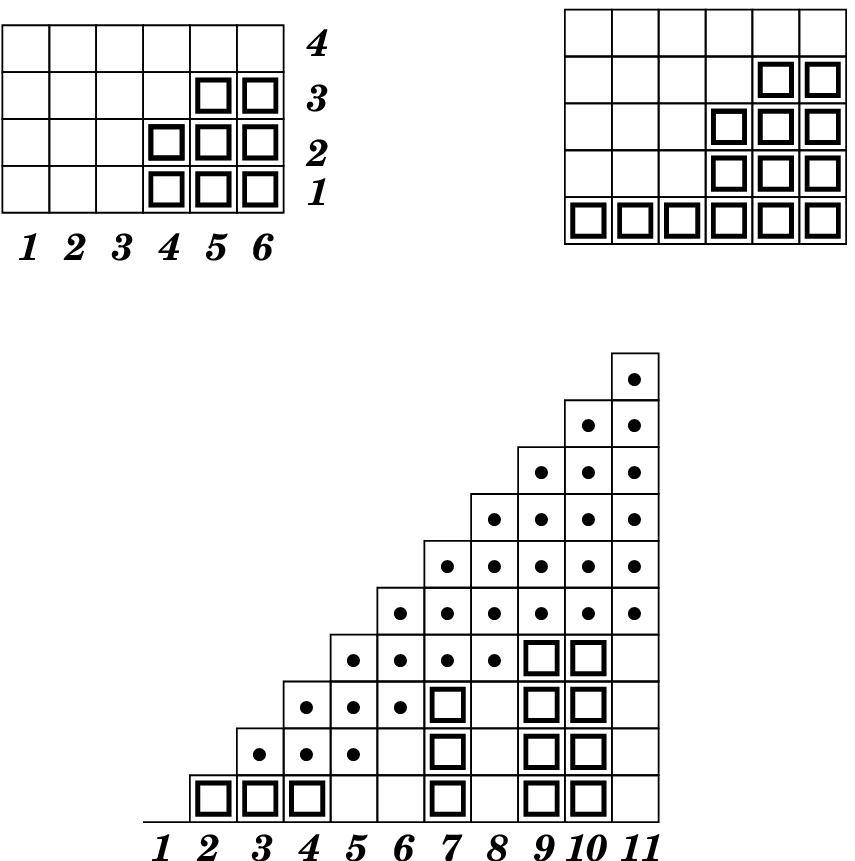}{Reconstructing of an element 
$P \in \mathcal{NT}_{n-k}(B_n)$ from a partitions contained in the  
$(n-k) \times (k-1)$ rectangle.}

Next we consider (\ref{eq:id6}). In this case, the Fibonacci 
rook tilings 
$P \in \mathcal{NT}_{n-k}(B_n)$ that we must consider 
are those tilings which have exactly one tile of height 2 and 
the rest of tiles must be of height 1. 
We let $c$ denote the column which contains 
the tile of height $2$.  We shall classify 
such tilings by the number $s$ of tiles of height 
1 that are in column $c$. 
Since the maximum number of non-canceled cells in any column 
is $k$ and every column which is tiled has a tile of height 
1 at the bottom of the column, $s$ can vary from 1 to $k-2$.
We shall think of the factor $q^{n-k}$ that sits outside 
of the first sum in  (\ref{eq:id6}) as the contribution 
from the tiles of height $1$ at the bottom of the $n-k$ columns 
that have tilings. The factor $sq^{s-1}$ accounts for the 
factor that comes from the column $c$. That is,  
there are $s-1$ tiles of height 1 in $c$ other than the tile of height 
1 at the bottom of column $c$  and the number of tiles of height 
1 that can lie below the tile of height of $2$ in $c$ can vary from 
1 to $s$.  

We interpret the $i$ in the inner sum as the number of columns to 
the right of $c$ which have tiles. There 
are $(s+1)i$ tiles of height 1 in rows $2$ through $s+2$ in each of 
these $i$ columns which account for the factor $q^{i(s+1)}$. 
Next we consider the set partition $\beta$ induced by 
that tiles above row $s+2$ that lie in these $i$ columns. 
$\beta$ must be contained in the $i \times (k-s-2)$ since 
the maximum number of uncanceled cells in any column is $k$. 
As $\beta$ varies over all possible partitions contained in 
$i \times (k-s-2)$, we get a factor of $\qbinom{i+k-s-2}{i}$. 
Finally we let $\alpha$ be the partition induced by the 
tilings in the columns to the left of column $c$ minus the 
tiles of height 1 at the bottom of these columns.  There 
are $n-k-i-1$ such columns and the maximum number of 
tiles in any such column is $s+1$.  As 
$\alpha$ varies over all possible partitions contained in 
$(n-k-i-1) \times (s+1)$, we get a factor of $\qbinom{s+n-k-i}{s+1}$.

The decomposition of a $P \in \mathcal{NT}_{n-k}(B_n)$ where 
$n =17$, $s=2$, $k=7$, and $i=3$ into the partitions $\alpha$, $\beta$, and 
the $i\times (s+1)$ rectangle is pictured in Figure \ref{fig:Sfnkp1}.

\fig{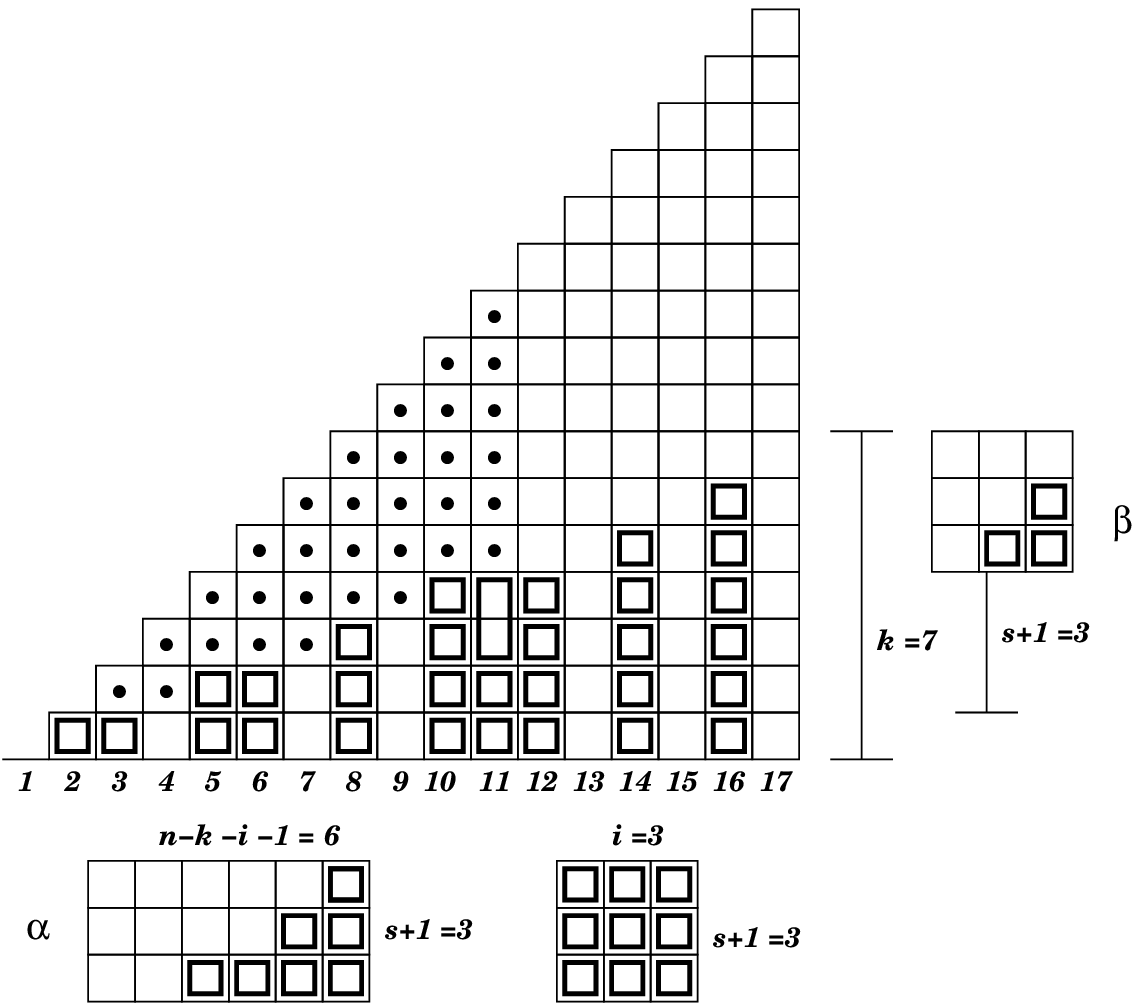}{The decomposition of $P \in \mathcal{NT}_{n-k}(B_n)$.}

We can use the same argument that we did in the 
proof of (\ref{eq:id5}) to prove that we can reconstruct 
a $P \in \mathcal{NT}_{n-k}(B_n)$ from $s$, $i$, and the partitions 
$\alpha$ and $\beta$.  

Thus we have proved that 
$$ \mathbf{Sf}_{n,k}(p,q)|_{p^1} =
q^{n-k} \sum_{s=1}^{k-2} sq^{s-1}\sum_{i=0}^{n-k-1} q^{i(s+1)}
\qbinom{i+k-s-2}{i} \qbinom{s+ n-k-i}{s+1}.$$ 

\end{proof}

We claim that (\ref{eq:id6}) is just a $q$-analogue of 
$\binom{k-1}{2}\binom{n-1}{k}$.  That is, we have 
the following theorem. 

\begin{theorem} For all $n \geq k \geq 3$, 
\begin{equation}\label{eq:Y}
 \mathbf{Sf}_{n,k}(p,1)|_{p^1} = \binom{k-1}{2}\binom{n-1}{k}
\end{equation}
\end{theorem}
\begin{proof}
We can easily prove (\ref{eq:Y}) by induction on $k$ and then 
by induction on $n$. First observe 
$\mathbf{Sf}_{k,k}(p,1)|_{p^1} =0$ since $\mathbf{Sf}_{k,k}(p,q) =1$. 
Thus for each $k$, the base case of the induction on $n$ holds. 

Next observe that for $k \geq 3$, $F_k(p,1)|_{p^1} =k-2$ and 
 $F_k(p,1)|_{p^0} =1$.  Hence 
\begin{eqnarray}\label{eq:YY}
\mathbf{Sf}_{n+1,k}(p,1)|_{p^1} &=& 
\mathbf{Sf}_{n,k-1}(p,1)|_{p^1} + 
\left(F_k(p,1)\mathbf{Sf}_{n,k}(p,1)\right)|_{p^1} \nonumber \\
&=& 
\mathbf{Sf}_{n,k-1}(p,1)|_{p^1} + 
F_k(p,1)|_{p^1}\mathbf{Sf}_{n,k}(p,1)|_{p^0} + 
F_k(p,1)|_{p^0}\mathbf{Sf}_{n,k}(p,1)|_{p^1} \nonumber \\
&=& \mathbf{Sf}_{n,k-1}(p,1)|_{p^1} + 
(k-2)\binom{n-1}{k-1} + \mathbf{Sf}_{n,k}(p,1)|_{p^1}.
\end{eqnarray}

Now suppose that $k=3$. Then we need to show 
that $\mathbf{Sf}_{n,3} = \binom{n-1}{3}$.
Note that for all 
$n \geq 2$, $\mathbf{Sf}_{n,2}(p,q) =q^{n-2}[n-1]_q$ so 
that $\mathbf{Sf}_{n,2}(p,1)|_{p^1} =0$. Thus using 
(\ref{eq:YY}) and induction, we see that 
\begin{equation*}
\mathbf{Sf}_{n+1,3}(p,1)|_{p^1} = 
\binom{n-1}{2}+\binom{n-1}{3} = \binom{n}{3}.
\end{equation*}
This establish (\ref{eq:Y}) in the case $k=3$. 

For $k > 3$, assume by induction 
that  $\mathbf{Sf}_{n,k-1} = \binom{k-1}{2}\binom{n-1}{k-1}$.
Then using 
(\ref{eq:YY}) and induction, we see that 
\begin{eqnarray*}
\mathbf{Sf}_{n+1,3}(p,1)|_{p^1} &=& \binom{k-1}{2}\binom{n-1}{k-1}
+(k-2)\binom{n-1}{k-1} +\binom{k-1}{2}\binom{n-1}{k} \\
&& \binom{k-1}{2}\binom{n-1}{k-1} +\binom{k-1}{2}\binom{n-1}{k} \\
&=& \binom{k-1}{2}\binom{n}{k}.
\end{eqnarray*}
\end{proof}

A sequence of real numbers $a_0, \ldots, a_n$ is is said to be {\em unimodal} 
if there is a $0 \leq j \leq n$ such that 
$a_0 \leq  \cdots \leq a_j \geq a_{j+1} \geq \cdots \geq a_n$ and 
is said to be {\em log concave} if for $0 \leq i \leq n$, 
$a_i^2- a_{i-1}a_{i+1} \geq 0$ where we set $a_{-1} = a_{n+1} =0$.
If a sequence is log concave, then 
it is unimodal. A polynomial $P(x) = \sum_{k=0}^n a_k x^k$ is said to be {\em unimodal} if $a_0, \ldots, a_n$ is a unimodal sequence and is said 
to be log concave if $a_0, \ldots, a_n$ is log concave. 
Computational evidence suggests that 
the polynomials $\mathbf{Sf}_{n,k}(p,1)$ are log concave for 
all $n \geq k$.  We can prove this for $k \leq 4$.  Clearly, this is true for 
$k=1$ and $k=2$ because by Theorem \ref{thm:id1} both 
$S_{n,1}(p,1)$ and $S_{n,2}(p,1)$ are just constants. For 
$k=3$ and $k=4$, we have the following theorem. 
\begin{theorem}
For all $n \geq 3$ and $s \geq 0$, 
\begin{equation}\label{Sfn3q=1}
\mathbf{Sf}_{n,3}(p,1)|_{p^s} = \binom{n-1}{s+2}.
\end{equation}
and for all $n \geq 4$ and $s \geq 0$, 
\begin{equation}\label{Sfn4q=1}
\mathbf{Sf}_{n,4}(p,1)|_{p^s} = (2^{s+1}-1)\binom{n-1}{s+3}.
\end{equation}
\end{theorem}
\begin{proof}
By Theorem \ref{thm:id5}, we know that $\mathbf{Sf}_{n,3}(p,1)|_{p^0} = 
\binom{n-1}{2}$ and 
$\mathbf{Sf}_{n,4}(p,1)|_{p^0} = \binom{n-1}{3}$. 
Thus our formulas hold for $s=0$. 

We then proceed first by induction on $n$ and then by induction on $s$. We know by Theorem \ref{thm:id1} 
that for $s \geq 1$, $\mathbf{Sf}_{n,2}|_{p^s} = 0$. Thus for 
$s \geq 1$, 
\begin{eqnarray*}
\mathbf{Sf}_{n,3}(p,1)|_{p^s} &=& \mathbf{Sf}_{n-1,2}(p,1)|_{p^s} + 
(F_3(p,1)\mathbf{Sf}_{n-1,3}(p,1))|_{p^s} \\
&=& ((1+p)\mathbf{Sf}_{n-1,3}(p,1))|_{p^s}\\
&=& \mathbf{Sf}_{n-1,3}(p,1))|_{p^s} + \mathbf{Sf}_{n-1,3}(p,1))|_{p^{s-1}} \\
&=& \binom{n-2}{s+2}+ \binom{n-2}{s+1} = \binom{n-1}{s+2}.
\end{eqnarray*}
Similarly, 
\begin{eqnarray*}
\mathbf{Sf}_{n,4}(p,1)|_{p^s} &=& \mathbf{Sf}_{n-1,3}(p,1)|_{p^s} + 
(F_4(p,1)\mathbf{Sf}_{n-1,4}(p,1))|_{p^s} \\
&=& \mathbf{Sf}_{n-1,3}(p,1)|_{p^s} +  
((1+2p)\mathbf{Sf}_{n-1,4}(p,1))|_{p^s}\\
&=& \mathbf{Sf}_{n-1,3}(p,1)|_{p^s} +\mathbf{Sf}_{n-1,4}(p,1))|_{p^s} + 2
\mathbf{Sf}_{n-1,4}(p,1))|_{p^{s-1}} \\
&=& \binom{n-2}{s+2} +(2^{s+1}-1)\binom{n-2}{s+3}+ 2(2^s-1)\binom{n-2}{s+2} \\
&=& (2^{s+1}-1) \left(\binom{n-2}{s+2} + \binom{n-1}{s+3}\right) \\
&=& (2^{s+1}-1) \binom{n-1}{s+3}.
\end{eqnarray*}
\end{proof}

It is then easy to prove by direct calculation 
that the sequences $\{\binom{n-1}{s+2}\}_{s > 0}$ and \\
$\{(2^s-1)\binom{n-1}{s+3}\}_{s > 0}$ are log concave. 
It is not obvious how this direct approach can be extended to 
prove that the polynomials $\mathbf{Sf}_{n,k}(p,1)$ are log concave 
for $k \geq 5$ because the formulas for $\mathbf{Sf}_{n,k}(p,1)|_{p^s}$
become more complicated.  For example, the following formulas are 
straightforward to prove by induction: 
\begin{eqnarray*}
\mathbf{Sf}_{n,5}(p,1)|_{p} &=& 6\binom{n-1}{5},\\
\mathbf{Sf}_{n,5}(p,1)|_{p^2} &=& 25\binom{n-1}{6}+\binom{n-1}{5},\\
\mathbf{Sf}_{n,5}(p,1)|_{p^3} &=& 90\binom{n-1}{7}+9\binom{n-1}{6}, \ \mbox{and}\\
\mathbf{Sf}_{n,5}(p,1)|_{p^4} &=& 301\binom{n-1}{8}+52\binom{n-1}{7} + 
\binom{n-1}{6}.
\end{eqnarray*}

Our next two theorems concern some results on 
$\mathbf{cf}_{n,1}(p,q)|_{p^i}$ and 
$\mathbf{cf}_{n,2}(p,q)|_{p^i}$ for small values of $i$. We 
start by considering $\mathbf{cf}_{n,1}(p,q)|_{p^i}$ for 
$i =0,1,2,3$.  

\begin{theorem} 
\begin{equation}\label{eq:cfn1p0}
\mathbf{cf}_{n,1}(p,q)|_{p^0} =q^{\binom{n}{2}} \ \mbox{for all} \ n \geq 1.
\end{equation}
\begin{equation}\label{eq:cfn1p1}
\mathbf{cf}_{n,1}(p,q)|_{p^1} =\begin{cases} 0 & \mbox{for} \ n=1,2,3 \\
\binom{n-2}{2} q^{\binom{n}{2}-2} & \ \mbox{for all} \ n \geq 4.
\end{cases}
\end{equation}
\begin{equation}\label{eq:cfn1p2}
\mathbf{cf}_{n,1}(p,q)|_{p^2} =\begin{cases} 0 & \mbox{for} \ n=1,2,3,4 \\
\left( 3\binom{n-1}{4} - \binom{n-3}{2}\right) q^{\binom{n}{2}-4} & \ \mbox{for all} \ n \geq 5.
\end{cases}
\end{equation}
\begin{equation}\label{eq:cfn1p3}
\mathbf{cf}_{n,1}(p,q)|_{p^3} =\begin{cases} 0 & \mbox{for} \ n=1,2,3,4,5 \\
\left(15\binom{n}{6}-6\binom{n-2}{4}+\binom{n-4}{4}\right)q^{\binom{n}{2}-6} & \ \mbox{for all} \ n \geq 5.
\end{cases}
\end{equation}
\end{theorem}
\begin{proof}
Equation (\ref{eq:cfn1p0}) follows from the fact that 
$\mathbf{cf}_{n,1}(p,q)|_{p^0}$ counts the Fibonacci file 
tiling of $B_n$ where all but the first column are filled with tiles 
of size 1. There are clearly $1+2+ \cdots +(n-1) = \binom{n}{2}$ such tiles. 

$\mathbf{cf}_{n,1}(p,q)|_{p^1}$ in (\ref{eq:cfn1p1}) counts the $(p,q)$-Fibonacci file 
tilings of $B_n$ where all but the first column are filled with tiles 
of size 1 except for one column $c \in \{4, \ldots, n\}$ which 
is tiled with $c-3$ tiles of height 1 and one tile of height 2. 
Thus the total number of tiles of height 1 in any such tiling 
is $q^{\binom{n}{2}-2}$ and the number of ways to tile column $c$ is 
$c-3$ depending on how many tiles of height 1 in $c$ lies below 
the tile of height 2 in $c$.   Thus if $n \geq 4$, there 
are $\sum_{c=4}^n (c-3) = \binom{n-2}{2}$ such file tilings.  Thus  
$\mathbf{cf}_{n,1}(p,q)|_{p^1}=0$ if $n \leq 3$ and 
$\mathbf{cf}_{n,1}(p,q)|_{p^1}=\binom{n-2}{2} q^{\binom{n}{2}-2}$ if 
$n \geq 4$.

$\mathbf{cf}_{n,1}(p,q)|_{p^2}$ in (\ref{eq:cfn1p2}) counts the 
$(p,q)$-Fibonacci file 
tilings of $B_n$ where there are tilings in columns $2, \ldots, n$ 
and we use exactly 
two tiles of height 2. Clearly, there are no such tilings for 
$n =1,2,3,4$.  For $n=5$, there are two such tilings 
which are pictured in Figure \ref{fig:c51p2}. Thus our 
formula holds for $n =5$.

\fig{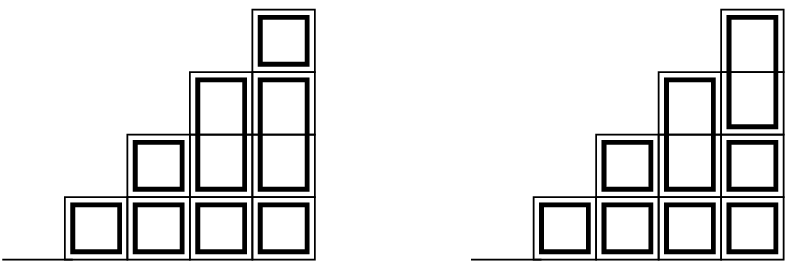}{The two file tiling for $\mathbf{cf}_{5,1}|_{p^2}$.} 

For $n > 5$, we proceed by induction. Note that 
$F_n(p,q)|_{p^2} = \binom{n-3}{2}$.  That is, for $F_n(p,q)|_{p^2}$, 
we are considering  Fibonacci tilings of height $n$ 
where we have $n-4$ tiles of height 1 
and two tiles of height 2. Since we must start with a tile 
of height 1, the number of such tilings is the number of 
rearrangement of $1^{n-5}2^2$ which is $\binom{n-3}{2}$. 
It is also easy to see that $F_n(p,q)|_{p} = n-2$.
Then 
\begin{eqnarray*}
\mathbf{cf}_{n,1}|_{p^2} &=&  \mathbf{cf}_{n-1,0}|_{p^2} + 
\left( F_{n-1}(p,q)\mathbf{cf}_{n-1,0}\right)|_{p^2}  \\
&=&\left( F_{n-1}(p,q)|_{p^2}\right) \left( \mathbf{cf}_{n-1,1}|_{p^0}\right)
+\left( F_{n-1}(p,q)|_{p^1}\right) \left( \mathbf{cf}_{n-1,1}|_{p^1}\right)
+ \\
&&\left( F_{n-1}(p,q)|_{p^0}\right) \left( \mathbf{cf}_{n-1,1}|_{p^2}\right)\\
&=&\left(\binom{n-4}{2}\right) \left(q^{\binom{n-1}{2}}\right)
+\left((n-3)q^{n-3}\right) \left( \binom{n-3}{2} q^{\binom{n-1}{2}-2}\right)
+ \\
&&\left(q^{n-1}\right) \left( q^{\binom{n-1}{2}-4}\left(3\binom{n-2}{4} - \binom{n-4}{2}\right)\right)\\
&=& q^{\binom{n}{2}-4}\left(\binom{n-4}{2} + (n-3)\binom{n-3}{2} + 
3\binom{n-2}{4} - \binom{n-4}{2}\right)\\
&=& q^{\binom{n}{2}-4}\left(3\binom{n-2}{3} - \binom{n-3}{2} + 3\binom{n-2}{4}\right)\\
&=& q^{\binom{n}{2}-4}  \left( 3 \binom{n-1}{4}-\binom{n-3}{2}  \right).
\end{eqnarray*}

$\mathbf{cf}_{n,1}(p,q)|_{p^3}$  in (\ref{eq:cfn1p3}) counts the Fibonacci file 
tilings of $B_n$ where there are tilings in columns $2, \ldots, n$  
and we use exactly 
three tiles of height 2. In general, the number of 
tiles of height 1 in such tiling is $\binom{n-2}{2}-6$.  
It is easy to check that there are no such tilings for 
$n =1,2,3,4,5$.  For $n=6$, there are nine such tilings 
which are pictured in Figure \ref{fig:c61p3}. Thus our 
formula holds for $n =6$. 

\fig{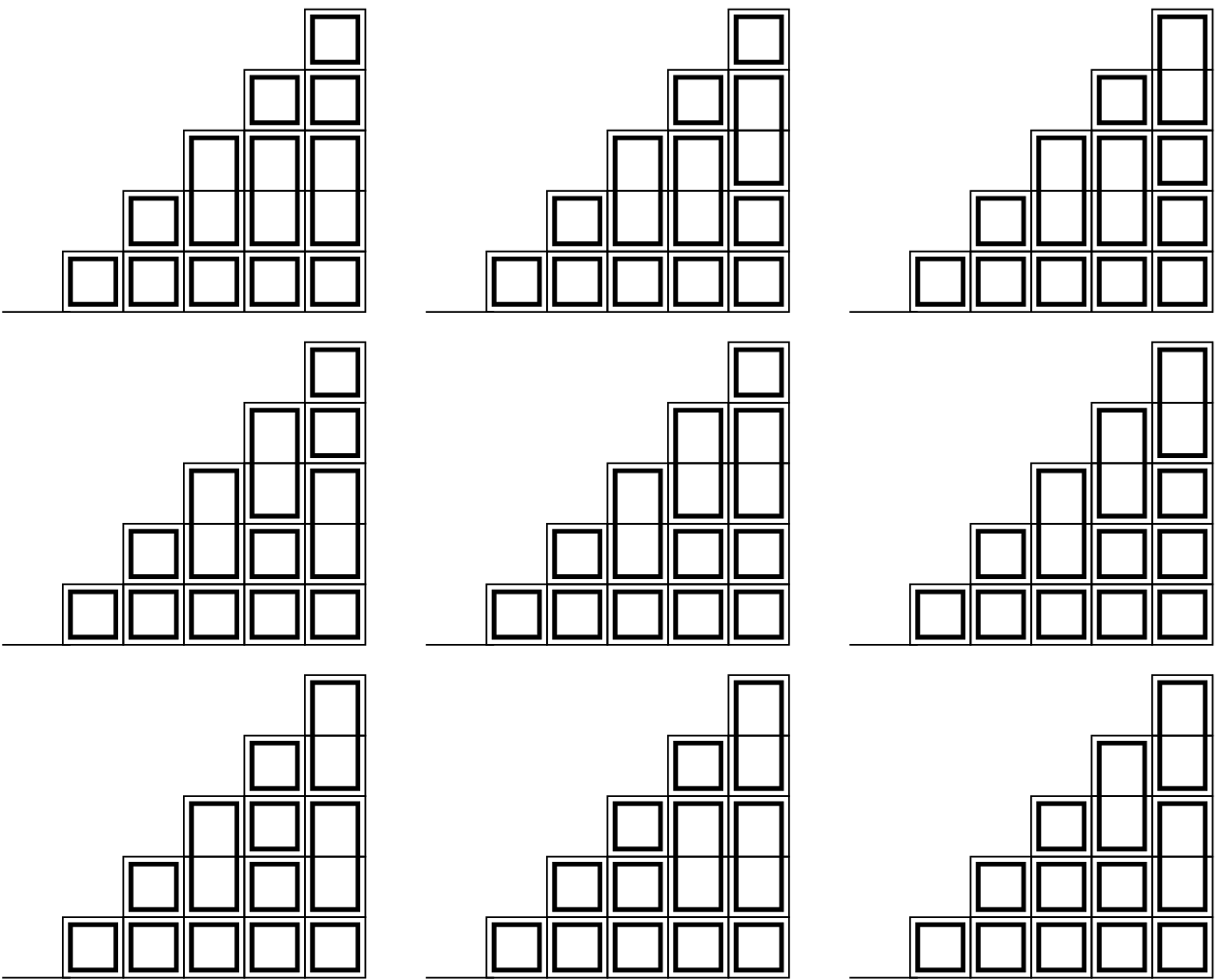}{The 9 file tiling for $\mathbf{cf}_{6,1}|_{p^3}$.} 

For $n > 6$, we proceed by induction. Note that 
$F_n(p,q)|_{p^3} = \binom{n-4}{3}$.  That is, for $F_n(p,q)|_{p^3}$, 
we are considering  Fibonacci tilings of height $n$ 
where we have $n-6$ tiles of height 1 
and three tiles of height 2. Since we must start with a tile 
of height 1, the number of such tilings is the number of 
rearrangement of $1^{n-7}2^3$ which is $\binom{n-4}{3}$. 
Then 
\begin{eqnarray*}
\mathbf{cf}_{n,1}|_{p^3} &=&  \mathbf{cf}_{n-1,0}|_{p^3} + 
\left( F_{n-1}(p,q)\mathbf{cf}_{n-1,1}\right)|_{p^3}  \\
&=&\left( F_{n-1}(p,q)|_{p^3}\right) \left( \mathbf{cf}_{n-1,1}|_{p^0}\right)
+\left( F_{n-1}(p,q)|_{p^2}\right) \left( \mathbf{cf}_{n-1,1}|_{p^1}\right)
+ \\
&& \left( F_{n-1}(p,q)|_{p^1}\right) \left( \mathbf{cf}_{n-1,1}|_{p^2}\right)
+\left( F_{n-1}(p,q)|_{p^0}\right) \left( \mathbf{cf}_{n-1,1}|_{p^3}\right)\\
&=&\left(\binom{n-5}{3}q^{n-7}\right) \left(q^{\binom{n-1}{2}}\right)
+\left(\binom{n-4}{2}q^{n-5}\right) \left( \binom{n-3}{2} q^{\binom{n-1}{2}-2}\right)
+ \\
&& \left((n-3)q^{n-3}\right)\left( q^{\binom{n-1}{2}-4}
\left(3\binom{n-2}{4} - \binom{n-4}{2}\right)\right) +
\left(q^{n-1}\right) \mathbf{cf}_{n-1,1}|_{p^3}\\
&=& q^{n-1} \mathbf{cf}_{n-1,1}|_{p^3} + \\
&& q^{\binom{n}{2}-6} \left( \binom{n-5}{3} + \binom{n-4}{2} \binom{n-3}{2} +(n-3)\left( 
3\binom{n-2}{2}-\binom{n-4}{2}\right) \right) .
\end{eqnarray*}
\\
It follows that we have the recursion 
\begin{multline}\label{cfn1p3rec}
\left(\mathbf{cf}_{n,1}|_{p^3}\right)|_{q^{\binom{n}{2}-6}} = \\
\left(\mathbf{cf}_{n-1,1}|_{p^3}\right)|_{q^{\binom{n-1}{2}-6}} +
\binom{n-5}{3} + \binom{n-4}{2} \binom{n-3}{2} +(n-3)\left( 
3\binom{n-2}{2}-\binom{n-4}{2}\right).
\end{multline}
Iterating (\ref{cfn1p3rec}), it follows 
that for $n \geq 7$, 
\begin{multline*}
\mathbf{cf}_{n,1}|_{p^3} = \\
  q^{\binom{n}{2}-6}\left( 9 + \sum_{k=7}^n 
\binom{k-5}{3} + \binom{k-4}{2}\binom{k-3}{2} +(k-3)\left( 
3 \binom{k-2}{2}-\binom{k-4}{2}\right)\right) = \\
\binom{n-4}{2}\left( \frac{12+28n+n^2-6n^3+n^4}{24}\right) =
15\binom{n}{6} -6\binom{n-2}{4}+\binom{n-4}{4}
\end{multline*}
where we have used Mathematica to verify the last two equalities. 
\end{proof}

We note that the sequence 
$\{3\binom{n-1}{4} -\binom{n-3}{2}\}_{n \geq 5}$ starts out 
$$2,12,39,95,195,357,602,954, \ldots .$$ This is sequence A086602 in 
the OEIS \cite{OEIS}. This sequence does not have a combinatorial 
interpretation so that we have now given a combinatorial 
interpretation of this sequence. 

The sequence $\{15\binom{n}{6} -6\binom{n-2}{4} +\binom{n-4}{4}\}_{n \geq 6}$ starts out 
$$9,75,331,1055,2745,6209,12670,23886,42285,71115, \ldots .$$
This sequence 
does not appear in the OEIS.

Next we consider $\mathbf{cf}_{n,2}(p,q)|_{p^i}$ for 
$i =0,1$.  

\begin{theorem}\label{thm:cfn2}
For $n \geq 2$, 
\begin{equation}\label{cfn2p0}
\mathbf{cf}_{n,2}(p,q)|_{p^0} = q^{\binom{n-1}{2}}[n-1]_q.
\end{equation}

\begin{equation}\label{cfn2p1}
\mathbf{cf}_{n,2}(p,q)|_{p^1} = 
\begin{cases} q^2+q^3 & \mbox{for} \ n =4 \\
\binom{n-2}{2} q^{\binom{n}{2}-3} +
q^{\binom{n-1}{2}-2} \sum_{i=0}^{n-3} \left(\binom{n-3}{2}+i\right)q^i
& \mbox{for} \ n \geq 5.
\end{cases}
\end{equation}
\end{theorem}
\begin{proof}

For $\mathbf{cf}_{n,2}(p,q)|_{p^0}$, we know that $\mathbf{cf}_{2,2}(p,q) =1$
so that our formula holds for $n =2$.   For $n \geq 3$, we 
must count the weights of all the Fibonacci file tilings 
where we use no tiles of height 2 such that 
there is exactly one empty column. It is easy to see 
that if the empty column is at the end, the 
number of tiles of size 1 is $1+2+ \cdots + (n-2) = \binom{n-1}{2}$.
Then as the empty column moves right to left, 
we see that we replace a column with $i$ tiles of height 
1 by a column 
with $i+1$ tiles of height 1. This process is pictured in Figure \ref{fig:cfn2} for $n =6$. It follows that for $n \geq 3$, 
$$\mathbf{cf}_{n,2}(p,q) = q^{\binom{n-1}{2}}+q^{\binom{n-1}{2}+1} \cdots + 
q^{\binom{n-1}{2}+(n-2)} 
=q^{\binom{n-1}{2}}(1+q+ \cdots + q^{n-2}) = q^{\binom{n-1}{2}}[n-1]_q.$$

\fig{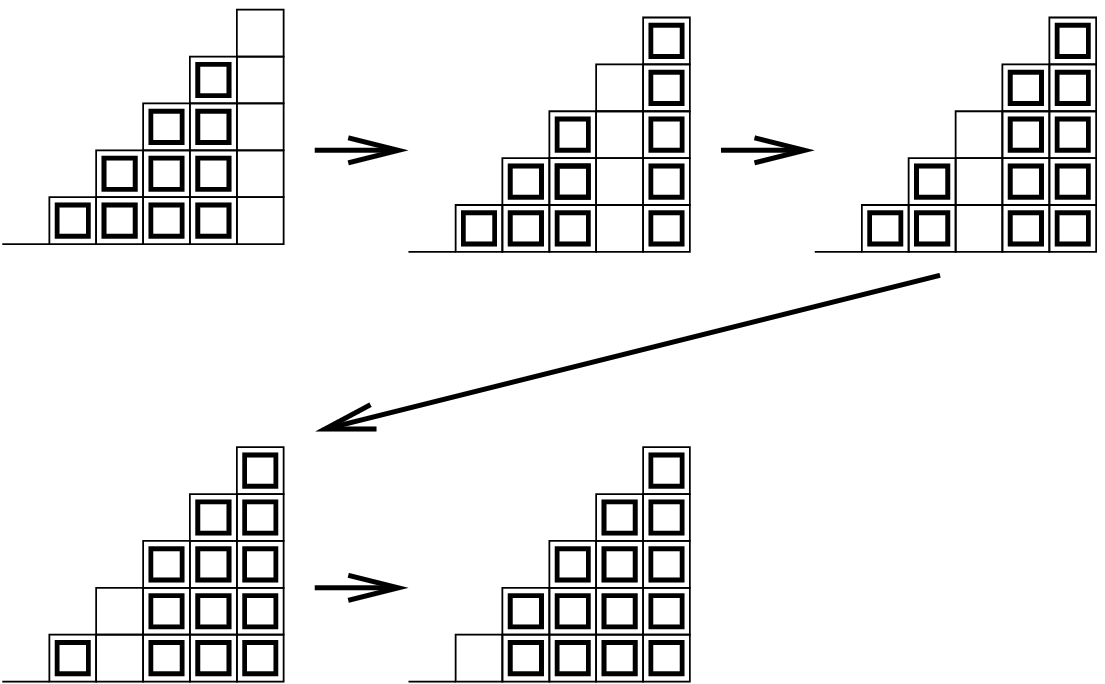}{The tilings for $\mathbf{cf}_{6,2}(p,q)$.}  

For $\mathbf{cf}_{4,2}(p,q)|_{p}$, 
there are only 
two Fibonacci file tilings which have one tile of height 2. 
These are pictured in 
Figure \ref{fig:cf42p1}. Thus 
$\mathbf{cf}_{4,2}(p,q)|_{p} =q^2+q^3$.

\fig{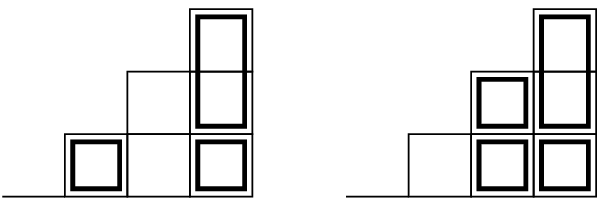}{The tilings for $\mathbf{cf}_{4,2}(p,q)|_{p}$.}

Note that by (\ref{eq:cfn1p1}) $\mathbf{cf}_{4,1}(p,q)|_{p}=q^4$. 
Hence 
\begin{eqnarray*}
\mathbf{cf}_{5,2}(p,q)|_{p} &=& 
\mathbf{cf}_{4,1}(p,q)|_{p} +\left(F_4(p,q)
\mathbf{cf}_{4,2}(p,q)\right)|_p \\
&=& q^{4} + \left(F_4(p,q)|_p\right) \left( 
\mathbf{cf}_{4,2}(p,q)|_{p^0}\right) + 
\left(F_4(p,q)|_{p^0}\right) \left( 
\mathbf{cf}_{4,2}(p,q)|_{p}\right)  \\
&=& q^4 + (2q^2)\left(\binom{3}{2}[3]_q\right) + q^4(q^2+q^3) \\
&=& q^4 + 2q^5 +3q^6 +3q^7. 
\end{eqnarray*}
This verifies our formula for $n =5$.

For $n > 5$, we proceed by induction. That is, 
\begin{eqnarray*}
\mathbf{cf}_{n,2}(p,q)|_{p} &=& 
\mathbf{cf}_{n-1,1}(p,q)|_{p} +\left(F_{n-1}(p,q)
\mathbf{cf}_{n-1,2}(p,q)\right)|_p \\
&=& \binom{n-3}{2}q^{\binom{n-1}{2}-2} + \\
&&\left(F_{n-1}(p,q)|_p\right) \left( 
\mathbf{cf}_{n-1,2}(p,q)|_{p^0}\right)+ 
\left(F_{n-1}(p,q)|_{p^0}\right) \left( 
\mathbf{cf}_{n-1,2}(p,q)|_{p}\right) \\
&=& \binom{n-3}{2}q^{\binom{n-1}{2}-2} + \left((n-3)q^{n-3}\right) \left( 
q^{\binom{n-2}{2}}[n-2]_q\right)+ \\
&& \left(q^{n-1} \right) \left(\binom{n-3}{2} q^{\binom{n-1}{2}-3} +
q^{\binom{n-2}{2}-2} \sum_{i=0}^{n-4} \left(\binom{n-4}{2}+i\right)q^i\right)\\
&=& \binom{n-3}{2}q^{\binom{n-1}{2}-2} +\\
&& q^{\binom{n-1}{2}-1}\left((n-3)q^{n-3}+\sum_{i=0}^{n-4}(n-3)q^i \right) +\\
&& \left(\binom{n-3}{2} q^{\binom{n}{2}-3} +
q^{\binom{n-1}{2}-1} \sum_{i=0}^{n-4} \left(\binom{n-4}{2}+i\right)q^i\right)\\
&=& \binom{n-3}{2}q^{\binom{n-1}{2}-2}+q^{\binom{n}{2}-3}\left((n-3) + \binom{n-3}{2} \right)+\\
&&q^{\binom{n-1}{2}-1} \sum_{i=0}^{n-4}
\left( (n-3) +\binom{n-4}{2} +i\right)q^i\\
&=& \binom{n-2}{2} q^{\binom{n}{2}-3} +
q^{\binom{n-1}{2}-2} \sum_{i=0}^{n-3} \left(\binom{n-3}{2}+i\right)q^i. 
\end{eqnarray*}
\end{proof}

It is easy to see from our formula for 
$\mathbf{cf}_{n,2}(p,q)|_{p}$ that 
\begin{eqnarray*}
\mathbf{cf}_{n,2}(p,1)|_{p} &=& \binom{n-2}{2} +
(n-2)\binom{n-3}{2} + \binom{n-2}{2}\\
&=& \binom{n-2}{2} +(n-4)\binom{n-2}{2} + \binom{n-2}{2}\\
&=& (n-2)\binom{n-2}{2}.
\end{eqnarray*}
We note that the sequence $\{\mathbf{cf}_{n,2}(p,q)|_{p}\}_{n \geq 4}$ 
starts out 
$$2,9,24,90,147,224,324,450,605,792,1014, \ldots .$$
This is sequence A006002 in the OEIS which does not have a combinatorial interpretation.  Thus we have given a combinatorial interpretation to this sequence.

Our computational evidence suggests that 
the polynomials $\mathbf{cf}_{n,k}(p,1)$ are also log concave. 
We can prove this in the case $k=1$. In that 
case, $\mathbf{cf}_{n,k}(p,1) = \prod_{i=1}^{n-1} F_i(p,1)$ and 
one can prove that the polynomials $F_n(p,1)$ have real roots. 
Thus  $\mathbf{cf}_{n,k}(p,1)$ has real roots and, hence, is 
log-concave.

\section{Conclusions}

In this paper, we studied Fibonacci analogues of 
the Stirling numbers of the first and second kind. That is, 
we have studied the connection  
coefficients defined by the equations
\begin{equation*}
x(x+F_1(p,q)) \cdots (x+F_{n-1}(p,q)) =
\sum_{k=1}^n \mathbf{cf}_{n,k}(p,q) x^k
\end{equation*}
and 
\begin{equation*}
x^n =
\sum_{k=1}^n \mathbf{Sf}_{n,k}(p,q) x(x-F_1(p,q))\cdots (x-F_{k-1}(p,q)).
\end{equation*}
We also have given a rook theory model for the $\mathbf{cf}_{n,k}(p,q)$s 
and $\mathbf{Sf}_{n,k}(p,q)$s. 

There are other natural $q$-analogues for Fibonacci analogues 
of the Stirling numbers of the first and second kind. 
For example, we could study the connection coefficients defined by the equations
\begin{equation*}
[x]_q[x+F_1]_q \cdots [x+F_{n-1}]_q =
\sum_{k=1}^n \mathbf{cF}_{n,k}(q) [x]_q^k
\end{equation*}
and 
\begin{equation*}
[x]_q^n =
\sum_{k=1}^n \mathbf{SF}_{n,k}(q) [x]_q[x-F_1]_q \cdots [x-F_{k-1}]_q.
\end{equation*}
It turns out that our basic rook theory model can 
also be used to give a combinatorial 
interpretation to $\mathbf{cF}_{n,k}(q)$'s 
and $\mathbf{SF}_{n,k}(q)$'s. In this case, 
we have to weight the Fibonacci tilings 
of height $n$ in a different way. The basic 
idea is that there is a natural 
tree associated with the Fibonacci tilings of height $n$. 
That is, we start from the top of a Fibonacci tiling 
and branch left if we see a tile of height 1 and branch 
right if we see a tiling of height 2. We shall call 
such a tree, the Fibonacci tree for $F_n$. 
For example, 
the Fibonacci tree for $F_5$ is pictured 
in Figure \ref{fig:Ftree}. 

\fig{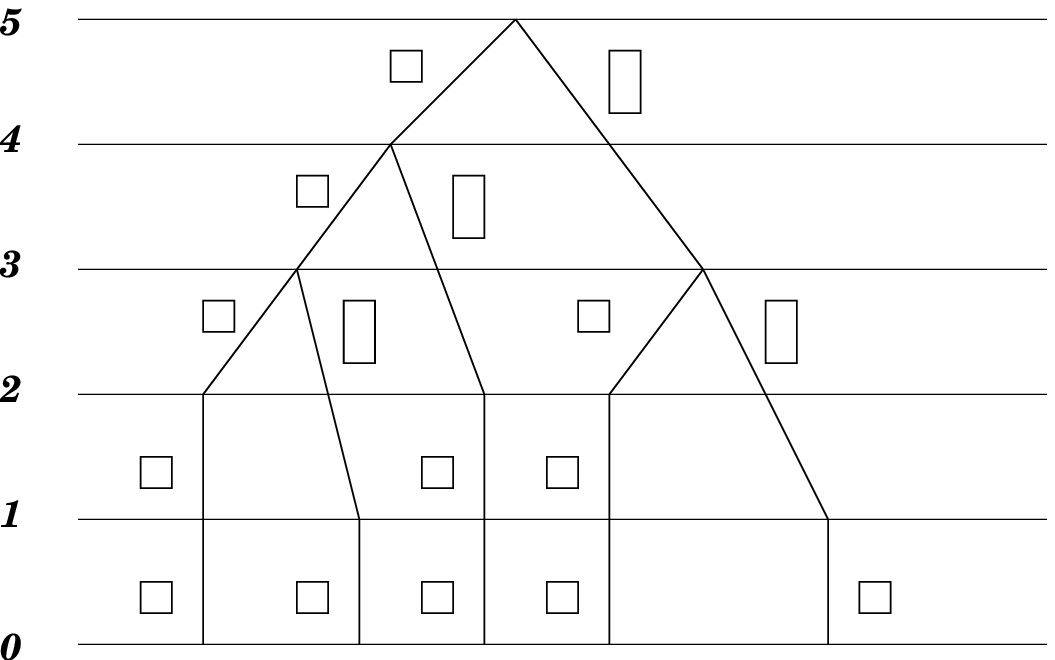}{The tree for $F_5$}

Then the paths in the tree correspond to Fibonacci tilings and 
we define the rank of a Fibonacci tiling $T$ of height $n$, 
$\mathbf{rank}(T)$, to be the number of paths in the 
Fibonacci tree for $F_n$ which lie to left of the path 
that corresponds to $T$.  In this way, the ranks of 
that Fibonacci tiling range from $0$ to $F_n-1$ and, hence 
$$\sum_{T \ \mbox{is a Fibonacci tiling of height} \ n}
q^{\mathbf{rank}(T)} = 1+q+ \cdots +q^{F_n-1} =[F_n]_q.$$

We shall show in a subsequent paper 
that by weighting a Fibonacci rook tilings or a Fibonacci 
file tilings of $B_n$, 
$$P = ((c_{i_1},T_{i_1}),  \ldots, (c_{i_k},T_{i_k}))$$ 
by $q^{s(P)} q^{\sum_{j=1}^k \mathbf{rank}(T_{i_j})}$ 
for some appropriate statistic $s(P)$, we 
can give a combinatorial interpretation to 
$\mathbf{cF}_{n,k}(q)$s 
and $\mathbf{SF}_{n,k}(q)$s.


\end{document}